\documentclass[a4paper,11pt,reqno]{article}

\usepackage[utf8]{inputenc}
\usepackage[T1]{fontenc}
\usepackage{lmodern}
\usepackage[english]{babel}
\usepackage{microtype}

\usepackage{amsmath,amssymb,amsfonts,amsthm}
\usepackage{mathtools,accents}
\usepackage{mathrsfs}
\usepackage{aliascnt}
\usepackage{braket}
\usepackage{bm}

\usepackage[a4paper,margin=3cm]{geometry}
\usepackage[citecolor=blue,colorlinks]{hyperref}

\usepackage{enumerate}
\usepackage{xcolor}

\makeatletter
\g@addto@macro\@floatboxreset\centering
\makeatother

\makeatletter
\def\newaliasedtheorem#1[#2]#3{
  \newaliascnt{#1@alt}{#2}
  \newtheorem{#1}[#1@alt]{#3}
  \expandafter\newcommand\csname #1@altname\endcsname{#3}
}
\makeatother

\numberwithin{equation}{section}

\newtheoremstyle{slanted}{\topsep}{\topsep}{\slshape}{}{\bfseries}{.}{.5em}{}

\theoremstyle{plain}
\newtheorem{theorem}{Theorem}[section]
\newaliasedtheorem{proposition}[theorem]{Proposition}
\newaliasedtheorem{lemma}[theorem]{Lemma}
\newaliasedtheorem{corollary}[theorem]{Corollary}
\newaliasedtheorem{conjecture}[theorem]{Conjecture}
\newaliasedtheorem{counterexample}[theorem]{Counterexample}

\theoremstyle{definition}
\newaliasedtheorem{definition}[theorem]{Definition}
\newaliasedtheorem{question}[theorem]{Question}
\newaliasedtheorem{example}[theorem]{Example}

\theoremstyle{remark}
\newaliasedtheorem{remark}[theorem]{Remark}

\let\altphi\phi
\let\phi\varphi
\let\varphi\altphi
\let\altphi\undefined

\newcommand{\di}{\mathop{}\!\mathrm{d}}

\DeclareMathOperator{\supp}{supp}

\newcommand{\Ch}{{\sf Ch}}

\newcommand{\haus}{\mathscr{H}}

\newcommand{\dist}{\mathsf{d}}

\newcommand{\meas}{\mathfrak{m}}

\DeclareMathOperator{\RCD}{RCD}

\DeclareMathOperator{\BE}{BE}
\newfont{\tmpf}{cmsy10 scaled 2500}

\newcommand{\intav}{{\mathop{\int\kern-10pt\rotatebox{0}{\textbf{--}}}}}

\renewcommand{\ }{\text{ }}
\def\<{\langle}
\def\>{\rangle}

\begin{document}
\title{From almost smooth spaces to RCD spaces
}
 
\author{
Shouhei Honda
\thanks{Graduate School of Mathematical Sciences, The University of Tokyo; \url{shouhei@ms.u-tokyo.ac.jp}}
\,and Song Sun
\thanks{Institute for Advanced Study in Mathematics, Zhejiang University; \url{songsun@zju.edu.cn}} }
\maketitle
\begin{abstract}
	We provide various characterizations for a given almost smooth space to be an RCD space, in terms of a local volume doubling and a local Poincar\'e inequality. Applications include a characterization of Einstein $4$-orbifolds.
	\end{abstract}
\tableofcontents
\section{Introduction}
\subsection{Motivation and background}

 An RCD \emph{condition}, or more precisely the \emph{$\RCD(K,N)$ condition} for two parameters $K$ and $N$, is a synthetic notion of lower bound on Ricci curvature and upper bound on dimension for metric measure spaces. Special examples of metric measure spaces verifying the $\RCD(K, N)$ condition, called \emph{$\RCD(K, N)$ spaces}, include \emph{Ricci limit spaces}, which are by definition pointed Gromov-Hausdorff limit spaces of complete Riemannian manifolds with Ricci curvature bounded below by $K$ and dimension bounded above by $N$. 

The structure theory of Ricci limit spaces has been extensively studied in the framework of the Cheeger-Colding theory \cite{CheegerColding, CheegerColding1, CheegerColding2, CheegerColding3} (see also, for instance, \cite{CN15, CJN} for the more recent update), which establishes a regular-singular decomposition in terms of tangent cone analysis via splitting techniques. This theory not only leads us to great resolutions in Riemannian geometry (for example to Anderson-Cheeger, Fukaya, Fukaya-Yamaguchi and Gromov's conjectures), but also particularly in the volume non-collapsing setting, has significant applications in various areas of geometry. Notably, their theory played a crucial role in the proof of the Yau-Tian-Donaldson conjecture for K\"ahler-Einstein metrics on Fano manifolds \cite{CDS15} and in the construction of the
K-moduli space of smoothable K-stable Fano varieties \cite{DS,  LWX19, Odaka15, SSY16}.  

The theory of $\RCD$ spaces, which can be regarded as the best synthetic treatment of Ricci limit spaces, has been developed, roughly speaking, in two ways. 
The first one is to use the \emph{Curvature-Dimension condition} \cite{LottVillani, Sturm06, Sturm06b} coming from the \emph{Optimal Transportation Theory}, together with the \emph{Riemannian assumption} called the \emph{infinitesimally Hilbertianity} proposed by \cite{AmbrosioGigliSavare14, Gigli1}.
The other one is to use the \emph{Bakry-\'Emery condition} coming from $\Gamma$-calculus based on the Dirichlet form theory. 
It is known from \cite{AGS, AmbrosioMondinoSavare, ErbarKuwadaSturm} that both approaches are the same, namely RCD spaces can be characterized/studied by completely different ways. It is worth mentioning that significant applications of the RCD theory to other geometry are already found, for instance, \cite{BMS} with \cite{KLP}, about an existence of infinitely many geodesics in Alexandrov geometry.

 Notice that the Cheeger-Colding theory is of a purely local feature but the RCD theory by definition requires  global information. It is thus an interesting question to give a local characterization of RCD spaces. In many cases of interest and in examples one deals with spaces that are \emph{almost smooth},  
 namely, roughly speaking, spaces given by the metric completion of a smooth Riemannian manifold with Ricci bounded below, whose singular set has a high Hausdorff codimension. The precise definition will be explained in subsection \ref{subsec:main} (and in Definition \ref{def:alsm} for more general weighted spaces).
 The problem then reduces to imposing suitable conditions on the singular set  (see also \cite{BKMR}). There are many geometric settings where almost smooth spaces play an important role, for example, the proof of the Hamilton-Tian conjecture \cite{CW, Ba16} and the study of K\"ahler-Einstein metrics on singular varieties (see for example, \cite{CCHSTT, S14, S, GuSo}).

In this paper we will give characterizations for almost smooth spaces, including weighted ones, to be $\RCD$ spaces. Our criterion will be in terms of a \emph{uniformly local} condition and the spaces are allowed to be \emph{non-compact}. 

In the next subsection let us explain what is the uniform local condition we will adopt.

\subsection{Uniformly local condition; PI}
Let $X=(X, \dist_X, \meas_X)$ be a \textit{metric measure space}, namely $(X, \dist_X)$ is a complete separable metric space and $\meas_X$ is a Borel measure on $X$ which is finite and positive on each open ball (Definition \ref{def:mm}). 
The uniformly local condition as mentioned above for $X$ is called \textit{uniformly locally PI}, where the terms of PI spaces appeared firstly in \cite{CK}. 
The definition is to verify a local volume doubling condition and a local $(1,2)$-Poincar\'e inequality. For our main purposes, since $X$ is eventually a proper geodesic space, the uniformly local PI condition can be stated simply as follows (see Proposition \ref{prop:poincaresobolev}); 
for any finite $r>0$, which plays a role of the radius of a ball, 
\begin{itemize}
    \item{(Local volume doubling)} we have $\meas_X(B_{2r}(x))\le C\meas_X(B_r(x))$;
    \item{(Local Poincar\'e inequality)} the first positive eigenvalue of the minus Laplacian associated with the Neumann boundary value condition on $B_r(x)$ is bounded below by $\frac{C^{-1}}{r^2}$,
\end{itemize}
where $C>1$ is a quantitative constant. Note that though the constant $C$ is allowed to depend on (more precisely, an upper bound of) the radius $r>0$ of the ball $B_r(x)$, we do \textit{not} assume the dependence on the center $x$. Therefore, in this sense, the above PI condition is \textit{uniform}, where we should be able to discuss even in the case when $C$ also depends on the center, see Question \ref{ques:1}.
For simplicity on our terminologies, let us say just ``PI'' instead of ``uniformly locally PI'' in the sequel. See Definition \ref{def:PI} for the precise definition.

There exist  many deep studies on PI metric measure spaces, including Sobolev embedding theorems and a Rademacher type theorem. 
We refer to \cite{BjornBjorn, HK, HKST, KLV} about this topic as nice textbooks.

Finally it should be mentioned that any $\RCD(K, N)$ space for finite $N$ is PI due to verifying the Bishop-Gromov inequality \cite{LottVillani, Sturm06b}, and a local $(1,1)$-Poincar\'e inequality \cite{Rajala}. 
\subsection{Main results}\label{subsec:main}
In order to introduce our main results, Theorems \ref{thm:maincod2} and \ref{thm:mainsmooth} below, let us prepare a few terminologies, where the precise descriptions will be explained later again.

Let $X=(X, \dist_X, \meas_X)$ be a metric measure space. 
An open subset $U$ of $X$ is said to be $n$-dimensionally \textit{smooth} if any point $x \in U$ has an open neighborhood $V$ of $x$ which is isometric, as metric measure spaces, to a smooth weighted Riemannian manifold of dimension $n$ (Definition \ref{def:smooth}). 
As already mentioned in the first subsection, our main targets are \textit{almost smooth} spaces defined as follows; $X$ is said to be $n$-dimensionally \textit{almost smooth} if $X$ is actually smooth in the sense above, except for a closed subset $\mathcal{S}$, called the \textit{singular set}, of null $2$-capacity (Definition \ref{def:alsm}). Then the complement $\mathcal{R}:=X\setminus \mathcal{S}$ is called the \textit{smooth part} (or the \textit{regular set}). 

It should be emphasized that examples of almost smooth spaces, including orbifolds as typical ones, can be found from other geometries as mentioned above, for example, subRiemannian geometry. See \cite{BMR, BT, DHPW, P, PW, RS} along this direction. 

Then the first main result of the paper is stated as follows. 
\begin{theorem}[Characterization of RCD for almost smooth space]\label{thm:maincod2}
    Let $X$ be an $n$-dimensional almost smooth metric measure space (thus recall, in particular, it is complete), whose metric structure is a length space. Then, for all $K \in \mathbb{R}$ and $N \in [n, \infty)$, the following two conditions are equivalent.
    \begin{enumerate}
        \item $X$ is an $\RCD(K, N)$ space;
        \item $X$ is PI with the Sobolev-to-Lipschitz property (SL) and a quantitative Lipschitz continuity of harmonic functions (written by QL for short), and the $N$-Bakry-\'Emery Ricci tensor is bounded below by $K$ on the smooth part.
    \end{enumerate}
\end{theorem}
We need to clarify the meaning of the QL, which is, roughly speaking, to verify
\begin{equation}
    |\nabla h| \le \frac{C}{r}\intav_B|h|\di \meas_X, \quad \text{on $\frac{1}{2}B$}
\end{equation}
for any harmonic function $h$ on an open ball $B$ in $X$, where $\frac{1}{2}B$ stands for the ball of the half radius with the same center of $B$, $\intav_B:=\frac{1}{\meas_X(B)}\int_B$, and $C$ is a quantitative positive constant. Thanks to a work in \cite{Jiang}, we know that any $\RCD(K, N)$ space for finite $N$ satisfies a QL. Thus our main contribution in the theorem above is to prove the implication from (2) to (1).

It should be emphasized that the QL-assumption \textit{cannot} be dropped because 
the cone over a large circle is \textit{globally} PI (see Remark \ref{rem:bilipinv}) and flat outside the pole, however it is not an $\RCD$ space. See Remark \ref{rem:example} for the details. In this sense, the theorem above is sharp.

On the other hand, in the case when the singular set $\mathcal{S}$ has a higher codimension, we can remove the QL-assumption as stated below. This is the second main result.
\begin{theorem}[Characterization of RCD under codimension $4$-singularity]\label{thm:mainsmooth}
    Let $X$ be an $n$-dimensional almost smooth metric measure space, whose metric structure is a length space. Assume that the singular set $\mathcal{S}$ has at least codimension $4$ in the sense;
    \begin{equation}\label{eq:codi4ass}
        \limsup_{\epsilon \to 0}\epsilon^{-4}\meas_X\left(B_{\epsilon}(\mathcal{S}\cap B)\right)<\infty,\quad \text{for any open ball $B$,}
    \end{equation}
    where $B_{\epsilon}(A)$ denotes the $\epsilon$-open neighborhood of $A$.
    Then, for all $K \in \mathbb{R}$ and $N \in [n, \infty)$, the following two conditions are equivalent.
    \begin{enumerate}
        \item $X$ is an $\RCD(K, N)$ space;
        \item $X$ is PI with the SL and the $N$-Bakry-\'Emery Ricci curvature is bounded below by $K$ on the smooth part.
    \end{enumerate}
\end{theorem}
It is worth pointing out that the two results above are also justified in a more non-smooth framework (Theorems \ref{thm:fromalmostrcdtorcd} and \ref{thm:cod4almost}) and that Theorem \ref{thm:mainsmooth} applies to give a characterization of Einstein orbifolds in dimension $4$ (Theorem \ref{thm:orbifold}).

Finally we should refer related papers \cite{Sturm18, STURM25} about characterizations of BE conditions. In particular compare \cite[Theorem 5.2]{STURM25} with Proposition \ref{prop:codi4}.

In the next subsection, let us provide outlines of the proofs of the results above.
\subsection{Proof and organization}
The technical key observation to realize the main results comes from works by \cite{JKY, Jiang, Jiang15}. Namely, under PI with a QL, we knew a quantitative local Lipschitz continuity for a solution of a Poisson equation (Theorem \ref{thm:Poisson}). In particular, the heat kernel is locally Lipschitz (Theorem \ref{thm:gradient}). 

Using these estimates, we can construct \textit{good} cut-off functions, originally going back to works by \cite{CheegerColding} for manifolds (see \cite{MN} for RCD spaces, where we follow the arguments in the later one \cite{MN}). Such cut-off functions play key roles to establish a local-to-global result (Theorem \ref{thm:LtoG}).

We are now in a position to explain how to prove Theorem \ref{thm:maincod2}. Fix $X$ verifying (2). By similar ideas obtained in \cite{Honda3}, with a new observation by \cite{S}, we can prove that the Bakry-\'Emery condition is satisfied \textit{locally} on $X$. Together with the above local-to-global result, we can improve this to the \textit{global} Bakry-\'Emery condition on $X$. Since the other conditions for $X$ to be an RCD space are easily justified only by our PI assumption, we see that $X$ is actually an $\RCD(K, N)$ space, thus (1) is satisfied. Namely we have Theorem \ref{thm:maincod2}. It is worth mentioning that as a new technical ingredient, we generalize a main result in \cite{Honda3} for compact spaces to the non-compact weighted case, see Theorem \ref{thm:be}.

In order to prove Theorem \ref{thm:mainsmooth}, it is enough to check that if $X$ satisfies (2) with (\ref{eq:codi4ass}), then a QL holds because of Theorem \ref{thm:maincod2}. To get the desired QL, we borrow a similar idea by \cite{Dai}. Namely, under assuming only PI, though a local Lipschitz continuity for a harmonic function $h$ does not hold in general (see Remark \ref{rem:example}), we know a H\"older continuity of $h$. Together with a quantitative Lipschitz continuity of $h$ on the smooth part $\mathcal{R}=X \setminus \mathcal{S}$ coming from the assumption on the $N$-Bakry-\'Emery Ricci curvature, we can get
\begin{equation}\label{eq:hn}
    |\nabla h|(x)\le \frac{C}{\dist_X(x, \mathcal{S})^{1-\gamma}}, \quad \text{for any $x \in \mathcal{R}$}
\end{equation}
for some quantitative $\gamma \in (0, 1)$. Then recalling our codimension $4$ assumption (\ref{eq:codi4ass}) together with the Bochner inequality on $\mathcal{R}$, (\ref{eq:hn}) implies that $|\nabla h|^2$ is in $H^{1,2}$ locally across $\mathcal{S}$. Finally applying the Bochner inequality again, we know that $|\nabla h|^2$ is subharmonic for the operator $\Delta-2K$. Then the standard Moser iteration techniques justified by PI together with the $H^{1,2}$-regularity of $|\nabla h|^2$ allow us to improve (\ref{eq:hn}) to being $|\nabla h|^2 \in L^{\infty}$ locally. This shows the desired QL, thus we have Theorem \ref{thm:mainsmooth}.
The non-smooth generalizations are also justified along the essentially same line.

The organization of the paper is as follows.

In the next section, Section \ref{sec:2}, we recall fundamental results on metric measure spaces with possibly self-contained proofs, and then we provide some technical results, including the construction of good cut-off functions as mentioned above.

In Section \ref{sec:3}, we discuss local notions on synthetic lower bounds of Ricci curvature from the points of views of the Bakry-\'Emery theory and the RCD theory. In particular we define \textit{almost $\RCD$ spaces} which allow us to generalize our almost smooth results to this non-smooth framework. 

Section \ref{sec:as} is devoted to the introduction on our main targets, almost smooth spaces with its fundamental properties. We also discuss on the vanishing $p$-capacity (for instance Proposition \ref{prop:codi4}).

We give proofs of the main results in Section \ref{sec:5}, and then the results for almost smooth spaces can be generalized to almost $\RCD$ spaces in Section \ref{sec:almostrcd}.

Finally, in the final section, Section \ref{sec:open}, we provide a couple of open problems along this direction.

\textbf{Acknowledgement.}
The first named author acknowledges support of the Grant-in-Aid for Scientific
Research (B) of 20H01799, the Grant-in-Aid for Scientific Research (B) of 21H00977
and Grant-in-Aid for Transformative Research Areas (A) of 22H05105. He also thanks Nicola Gigli and Karl-Theodor Sturm for their useful comments. The second named author thanks Jikang Wang and Xingyu Zhu for discussions on RCD spaces. The authors wish to thank Qin Deng and Kohei Suzuki for their pointing out on errors on the first version.

\textbf{Conflict of interest}: The authors declare that they have no conflict of interest.

\textbf{Data availability}: Data sharing is not applicable to this paper as no datasets were generated or analyzed during the current study.
\section{Geometric analysis on metric measure space}\label{sec:2}
Let us provide a list for our terminologies as follows.
\begin{itemize}
    \item For an open ball $B=B_r(x)$ centered at $x \in X$ of radius $r>0$ in a metric space $(X, \dist_X)$, put
$$
RB:=B_{Rr}(x),\quad \text{for any $R>0$.}
$$
We also denote by $\bar B=\{y \in X| \dist_X(x, y) \le r\}$ the closed ball.
\item We will use the standard notations for spaces of functions, for instance, $L^p, \mathrm{Lip}$, stand for the spaces of all $L^p$-functions, of all Lipschitz functions, respectively. 
\item For any space $F$ consisting with functions on $X$, $F_c$ denotes the space of functions, belonging to $F$, with compact supports. 
\item We always identify two objects under isomorphisms.
\item Whenever we will discuss analytic and geometric inequalities below, constants $C$ may be changed from line to line, quantitatively.
\end{itemize}

\subsection{Metric measure space}
Let us start this subsection by introducing the following.
\begin{definition}[Metric measure space]\label{def:mm}
We say that a triple $(X, \dist_X, \meas_X)$, denoted by $X$ for short in the sequel, is a \textit{metric measure space} if $(X, \dist_X)$ is a complete separable metric space and $\meas_X$ is a Borel measure on $X$ which is finite and positive on any open ball.
\end{definition}
\begin{remark}
Some references do not assume the completeness for a metric measure space. In any case, a map between metric measure spaces is called an (\textit{local}, respectively) \textit{isometry} if it is surjective, and it preserves the distance and the measures (on a neighborhood of any point, respectively).     
\end{remark}
In the sequel, let us fix a metric measure space $X$ and let us refer a nice textbook \cite{GP} (see also \cite{AmbrosioGigliSavare13}) covering the details in this subsection.

The $H^{1,2}$-Sobolev space on $X$ is defined as follows.
\begin{definition}[$H^{1,2}$-Sobolev space]
Let us define the \textit{$H^{1,2}$-Sobolev space} as follows.
\begin{enumerate}
    \item{(Cheeger energy)} The \textit{Cheeger energy} $\Ch:L^2(X) \to [0, \infty]$ is defined by
\begin{align}
    &\Ch(f) \nonumber \\
    &:=\inf_{\{f_i\}_i}\left\{ \liminf_{i \to \infty}\frac{1}{2}\int_X(\mathrm{Lip}f_i)^2\di \meas_X \Big| \|f_i \to f\|_{L^2} \to 0, f_i \in \mathrm{Lip} \cap L^{\infty}\cap L^2(X)\right\},
\end{align}
where $\mathrm{Lip}f(x)$ denotes the \textit{local slope at} $x \in X$ of a locally Lipschitz function $f$ defined by
\begin{equation}
    \mathrm{Lip}f(x):=\limsup_{y \to x}\frac{|f(x)-f(y)|}{\dist_X(x,y)}
\end{equation}
if $x$ is not isolated,  and $\mathrm{Lip}f(x):=0$ otherwise.
\item{(Sobolev space)} The \textit{Sobolev space} $H^{1,2}(X)$ is defined by:
\begin{equation}
    H^{1,2}(X):=\left\{ f \in L^2(X) \Big| \Ch(f)<\infty\right\}
\end{equation}
 which is a Banach space equipped with the norm $\|f\|_{H^{1,2}}:=\sqrt{\|f\|_{L^2}^2+2\Ch(f)}$.
\end{enumerate}
 \end{definition}
 The Cheeger energy of a Sobolev function can be written by the canonical object as follows. See also \cite[Proposition 2.2.8]{GP} for a variant and see \cite{AILP} for equivalent definitions of Sobolev spaces.
 \begin{definition}[Minimal relaxed slope]
     For any $f \in H^{1,2}(X)$, let us consider the set, denoted by $R_f$, of all $g \in L^2(X)$ (called \textit{relaxed slopes} of $f$) satisfying:
     \begin{itemize}
         \item there exist $h \in L^2(X)$ and a sequence $f_i \in \mathrm{Lip}\cap L^{\infty} \cap L^2(X)$ such that $\|f_i-f\|_{L^2}\to 0$, that $\mathrm{Lip}f_i$ $L^2$-weakly converge to $h$ and that $h \le g$ for $\meas_X$-a.e..
     \end{itemize}
     Then $R_f$ is closed and convex in $L^2(X)$. Denoting by $|\nabla f|$ the unique element of $R_f$ having the minimal norm in $L^2(X)$, we have
     \begin{equation}
         \Ch(f)=\frac{1}{2}\int_X|\nabla f|^2\di \meas_X.
     \end{equation}
     We call $|\nabla f|$ the \textit{minimal relaxed slope} of $f$.
 \end{definition}
 One of important properties on $|\nabla f|$ is the \textit{locality} in the sense (see \cite[Theorems 2.1.28 and 4.1.4]{GP}): 
 \begin{itemize}
     \item{(Locality of gradient)} for all $f, g \in H^{1,2}(X)$, we have
     \begin{equation}\label{eq:localg}
         |\nabla (f-g)|=0, \quad \text{for $\meas_X$-a.e. on $\{f=g\}$},
     \end{equation}
     in particular we know 
     \begin{equation}\label{eq:localitygrad}
         |\nabla f|=|\nabla g|, \quad \text{ for $\meas_X$-a.e. on $\{f=g\}$.}
     \end{equation}
 \end{itemize}
\begin{definition}[Inifinitesimally Hilbertian (IH)]
We say that $X$ is \textit{inifinitesimally Hilbertian}, written by \textit{IH} for short, if $H^{1,2}(X)$ is a Hilbert space.
\end{definition}
If $X$ is IH, then it is known (see \cite[Theorem 4.3.3]{GP}) that for all $f, g\in H^{1,2}(X)$,
\begin{equation}
|\nabla (f+g)|^2+|\nabla (f-g)|^2=2|\nabla f|^2+2|\nabla g|^2,\quad \text{$\meas_X$-a.e.,}
\end{equation}
namely the $L^1$-function defined by
\begin{equation}
    \langle \nabla f, \nabla g\rangle := \frac{1}{2}\left( |\nabla (f+g)|^2-|\nabla f|^2-|\nabla g|^2\right)
\end{equation}
plays a role as an inner product in $\meas_X$-a.e. sense, thus we can regard this as the \textit{canonical Riemannian metric} of $X$. 
\begin{remark}\label{rem:grad}
It is easy to see that if there exists a sequence of open subsets $U_i$ such that $\meas_X(X\setminus \bigcup_iU_i)=0$ and that each $H^{1,2}(U_i)$, defined in Definition \ref{def:localsob} later, is a Hilbert space, then, $X$ is IH.
\end{remark}
Let us define the \textit{Laplacian} in this framework (see \cite[Definition 5.2.1]{GP}).
\begin{definition}[Laplacian]\label{def:lap}
    Assume that $X$ is IH. Then define the \textit{domain of the Laplacian}, denoted by $D(\Delta)$, by the set of all function $f \in H^{1,2}(X)$ satisfying the following: 
    \begin{itemize}
        \item there exists $\phi \in L^2(X)$ such that 
    \begin{equation}
        \int_X\langle \nabla f, \nabla g\rangle \di \meas_X=-\int_X\phi g\di \meas_X,\quad \text{for any $g \in H^{1,2}(X)$.}
    \end{equation}
    \end{itemize}
    Since $\phi$ is unique, we denote it by $\Delta f$, called the \textit{Laplacian} of $f$.
\end{definition}
As a direct consequence of the locality of $|\nabla f|$, (\ref{eq:localg}), we know that the Laplacian also has the locality property in the sense:
\begin{itemize}
     \item{(Locality of Laplacian)} for all $f, g \in D(\Delta)$, if $f=g$ on an open subset $U$ of $X$, then we have 
     \begin{equation}\label{eq:localitylap}
         \Delta f=\Delta g,\quad \text{for $\meas_X$-a.e. on $U$}.
     \end{equation}
 \end{itemize}
Finally let us introduce the \textit{heat flow} in our non-smooth framework.
 \begin{definition}[Heat flow]
 Assume that $X$ is IH.
     For any $f \in L^2(X)$ the \textit{heat flow} $h_{\cdot}f:(0, \infty) \to L^2(X)$ is defined by the \textit{gradient flow} of the Cheeger energy, namely 
     \begin{enumerate}
         \item $h_{\cdot}f$ is an absolutely continuous curve;
         \item $h_tf \in D(\Delta)$ for any $t>0$;
         \item $h_tf \to f$ in $L^2(X)$ as $t \to 0$;
         \item we have
         \begin{equation}\label{eq:heateq}
             \frac{\di}{\di t}h_tf=\Delta h_tf,\quad \text{for $\mathcal{L}^1$-a.e. $t>0$.}
         \end{equation}
     \end{enumerate}
 \end{definition}
 We end this subsection by introducing fundamental facts on the heat flow.
 \begin{proposition}\label{prop:heatconv}
 We have the following.
 \begin{enumerate}
     \item The absolute continuity of $h_{\cdot}f$ can be improved to the smoothness with the validity of (\ref{eq:heateq}) for any $t>0$;
     \item If $f \in H^{1,2}(X)$, then $h_tf \to f$ in $H^{1,2}(X)$ as $t \to 0$;
     \item We have a priori estimates;
     \begin{equation}\label{eq:heatflowbound}
     \|h_tf\|_{L^2} \le \|f\|_{L^2}, \quad \Ch(h_tf) \le \frac{\|f\|_{L^2}^2}{2t}, \quad \|\Delta h_tf\|_{L^2}^2\le \frac{\|f\|_{L^2}^2}{t^2}.
 \end{equation}
 \end{enumerate}     
 \end{proposition}
 \begin{proof}
     See \cite[Proposition 5.2.12]{GP} for the proof of (1). For (2), since \cite[(6a) of Theorem 5.1.12]{GP} implies $\limsup_i\Ch(f_i) \le \Ch(f)$, we conclude. Finally see \cite[Remark 5.2.11]{GP} for the details of (3).
 \end{proof}
 \begin{remark}\label{rem:tensor}
     If $X$ is IH, we can discuss ($L^p$-)\textit{tensor fields} on $X$. See \cite{Gigli} for the details.
 \end{remark}
 \subsection{BE and RCD spaces}
 In this subsection, let us provide two notions of synthetic lower bounds on Ricci curvature. We refer \cite{A, GIGLI, STURM} for recent nice surveys about the topic.
The first one is the following, coming from the \textit{Bakry-\'Emery theory} \cite{BE} (see also \cite{AGS}).

In the sequel, let us fix a metric measure space $X$.
 \begin{definition}[BE space]
     We say that $X$ is a \textit{$\BE(K, N)$ space} for some $K \in \mathbb{R}$ and some $N \in [1, \infty]$, or a \textit{$\BE$ space} for short, if it is IH with the following:
     \begin{itemize}
         \item{(Bochner inequality)} for any $\phi \in D(\Delta) \cap L^{\infty}(X)$ with $\phi \ge 0$ and $\Delta \phi \in L^{\infty}(X)$, and for any $f \in D(\Delta)$ with $\Delta f \in H^{1,2}(X)$, we have
         \begin{equation}
             \frac{1}{2}\int_X\Delta \phi |\nabla f|^2\di \meas_X\ge \int_X\phi \left( \frac{(\Delta f)^2}{N}+\langle \nabla \Delta f, \nabla f\rangle +K|\nabla f|^2\right)\di \meas_X.
         \end{equation}
     \end{itemize}
 \end{definition}
 Then we are now ready to define RCD spaces.
 \begin{definition}[RCD space] 
     We say that $X$ is an \textit{$\RCD(K, N)$ space} for some $K \in \mathbb{R}$ and some $N \in [1 , \infty]$, or an \textit{$\RCD$ space} for short, if the following conditions are satisfied.
     \begin{enumerate}
     \item{(BE)} $X$ is a $\BE(K, N)$ space; 
         \item{(Volume growth)} 
         For some $x \in X$ and some $C>1$, we have 
         \begin{equation}\label{eq:volumegrowth}
             \meas_X(B_r(x))\le C\exp \left(Cr^2\right),\quad \text{for any $r>0$.}
         \end{equation}
         \item{(Sobolev-to-Lipschitz property {\color{blue}(SL)})} If a Sobolev function $f \in H^{1,2}(X)$ satisfies
         \begin{equation}
             |\nabla f| \le 1,\quad \text{for $\meas_X$-a.e.,}
         \end{equation}
         then $f$ has a $1$-Lipschitz representative.
     \end{enumerate}
 \end{definition}
 Next let us introduce a special class of $\RCD(K, N)$ spaces defined in \cite[Definition 1.1]{DG}.
 \begin{definition}[Non-collapsed RCD space]
 An $\RCD(K, N)$ space $X$ for some $K \in \mathbb{R}$ and some $N \in [1, \infty)$ is said to be \textit{non-collapsed} if $\meas_X=\haus^N$ holds, where $\haus^N$ denotes the Hausdorff measure of dimension $N$.
 \end{definition}
 Non-collapsed $\RCD(K, N)$ spaces have finer geometric/analytic properties rather than that of general $\RCD(K, N)$ spaces including a fact that $N$ must be an integer and it coincides with the \textit{essential dimension}, defined in Theorem \ref{thm:fundarcd}. See \cite[Theorem 1.12]{DG} (see also \cite{KM}). These finer properties will play important roles for possible applications of the RCD theory to other subjects, see \cite{S}.

 The following is due to \cite[Proposition 3.3.18 and Theorem 3.6.7]{Gigli}.
 \begin{theorem}[Hessian and its Bochner inequality]\label{thm:hessboch}
 Let $X$ be an $\RCD(K, \infty)$ space for some $K \in \mathbb{R}$. Then
 for any $f \in D(\Delta)$, the \textit{Hessian} of $f$, denoted by $\mathrm{Hess}_f$, is well-defined as an $L^2$-tensor of type $(0,2)$ (see Remark \ref{rem:tensor}), and it can be characterized by satisfying
 \begin{align}
     \mathrm{Hess}_f(\nabla g, \nabla h)=\frac{1}{2}\left(\langle \nabla \langle \nabla f, \nabla g\rangle, \nabla h\rangle +\langle \nabla \langle \nabla f, \nabla h\rangle, \nabla g\rangle -\langle \nabla \langle \nabla g, \nabla h\rangle, \nabla f\rangle \right)
 \end{align}
 for all \textit{test functions} $g, h$ defined by belonging to the set (see also Remark \ref{rem:test}): 
\begin{equation}\label{eq:testdef}
    \mathrm{Test}(X):=\{g \in D(\Delta) \cap L^{\infty} \cap \mathrm{Lip}(X) | \Delta g \in H^{1,2}(X)\}.
\end{equation}
Moreover 
 \begin{enumerate}
     \item we have \begin{align}
     \frac{1}{2}\int_X\Delta \phi |\nabla f|^2 \di \meas_X \ge \int_X \left( \phi|\mathrm{Hess}_f|^2 -(\Delta f)^2\phi-\Delta f\langle \nabla f, \nabla \phi \rangle+ K\phi |\nabla f|^2\right)\di \meas_X
 \end{align}
 for any $\phi \in D(\Delta) \cap L^{\infty}(X)$ with $\phi \ge 0$ and $\Delta \phi \in L^{\infty}(X)$ (then $|\nabla \phi| \in L^{\infty}(X)$), and for any $f \in D(\Delta)$. 
 Thus, if furthermore $\Delta f \in H^{1,2}(X)$ holds, then 
 \begin{align}\label{eq:hessbochner}
     \frac{1}{2}\int_X\Delta \phi |\nabla f|^2 \di \meas_X \ge \int_X \phi \left( |\mathrm{Hess}_f|^2 +\langle \nabla \Delta f, \nabla f\rangle+ K |\nabla f|^2\right)\di \meas_X;
 \end{align}
     \item we have for any $f \in D(\Delta)$
     \begin{equation}\label{eq:bine}
     \int_X|\mathrm{Hess}_f|^2\di \meas_X \le \int_X\left((\Delta f)^2-K|\nabla f|^2\right) \di \meas_X.
 \end{equation}
 \end{enumerate}
 \end{theorem}
 The following locality of the Hessian is a direct consequence of that of the gradient (\ref{eq:localitygrad}) (see also \cite[Proposition 3.3.24]{Gigli}).
 \begin{itemize}
     \item{(Locality of Hessian)} for all $f, g \in D(\Delta)$, we have 
     \begin{equation}\label{eq:localityhess}
         \mathrm{Hess}_f=\mathrm{Hess}_g,\quad \text{for $\meas_X$-a.e. on $\{f=g\}$}.
     \end{equation}
 \end{itemize}
 Finally let us mention the following. We refer \cite[Theorem 0.1]{BrueSemola} for (1), and \cite[Theorem 1.5]{BGHZ} (or \cite[Theorem 1.2]{H} in the compact case) together with \cite[Theorem 3.12]{BrueSemola} and \cite[Proposition 3.2]{Han2} for (2).
 \begin{theorem}\label{thm:fundarcd}
     Let us assume that $X$ is an $\RCD(K, N)$ space for some $K \in \mathbb{R}$ and some $N \in [1, \infty)$, which is not a single point. 
     \begin{enumerate}
         \item{(Essential dimension)} There exists a unique $n \in \mathbb{N}\cap [1, N]$ such that for $\meas_X$-a.e. $x \in X$, any tangent cone at $x$ is isometric to the Euclidean space of dimension $n$, namely
         \begin{equation}\label{eq:tangentcone}
             \left(X, \frac{1}{r}\dist_X, x, \frac{1}{\meas_X(B_r(x))}\meas_X\right) \stackrel{\mathrm{pmGH}}{\to} \left(\mathbb{R}^n, \dist_{\mathrm{Euc}}, 0_n, \frac{1}{\omega_n}\haus^n \right),\quad \text{as $r \to 0$},
         \end{equation}
         where $pmGH$ stands for the pointed measured Gromov-Hausdorff convergence. We say $n$ the essential dimension of $X$.
         \item{(Characterization of non-collapsed space)} Denoting by $n$ the essential dimension of $X$, $\meas_X=c\haus^n$ holds for some $c>0$ and $(X, \dist_X, \haus^n)$ is a non-collapsed $\RCD(K, n)$ space if and only if
         \begin{equation}\label{eq:hesstr}
             \mathrm{tr}(\mathrm{Hess}_f)=\Delta f
         \end{equation}
         holds for any $f \in D(\Delta)$.
     \end{enumerate}
 \end{theorem}
 Defining the $n$-dimensional (tangential) \textit{regular set} $\mathcal{R}=\mathcal{R}^{\mathrm{tang}}$ by the set of all points $x \in X$ satisfying (\ref{eq:tangentcone}), we knew some weak convexity of $\mathcal{R}$, see \cite[Theorem 6.5]{Deng}. The complement $\mathcal{S}=\mathcal{S}^{\mathrm{tang}}:=X \setminus \mathcal{R}$ is called the (tangential) \textit{singular set} (note that strictly speaking, this is different from the standard one, \cite[Definition 0.2]{CheegerColding1}).
 Note that similar, but different, notions of singular sets will appear in Definitions \ref{def:weighted} and \ref{def:almostrcd}.
 
 \subsection{Local objects}\label{subsec:localobj}
 Let us fix a metric measure space $X$ with an open subset $U$ of $X$.
 \begin{definition}[Local Sobolev spaces $H^{1,2}(U), H^{1,2}_0(U)$ and $H^{1,2}_{\mathrm{loc}}(U)$]\label{def:localsob}
 Let us define three \textit{local} Sobolev spaces as follows;
 \begin{enumerate}
     \item{($H^{1,2}_0(U)$)} define $H^{1,2}_0(U)$ by the closure of $\mathrm{Lip}_c(U)$ in $H^{1,2}(X)$; 
     \item{($H^{1,2}_{\mathrm{loc}}(U)$)} define $H^{1,2}_{\mathrm{loc}}(U)$ by the set of all $f:U \to \mathbb{R}$ satisfying that $\rho f \in H^{1,2}(X)$ for any $\rho \in \mathrm{Lip}_c(U)$ (namely, put $\rho f=0$ outside $U$); 
     \item{($H^{1,2}(U)$)} define $H^{1,2}(U)$ by the set of all $f \in H^{1,2}_{\mathrm{loc}}(U)$ satisfying $|f|+|\nabla f| \in L^2(U)$, where $|\nabla f|$ makes sense because of the locality of gradient, (\ref{eq:localg}).
 \end{enumerate}
\end{definition}
Note that both $H^{1,2}(U)$ and $H^{1,2}_0(U)$ are Hilbert spaces if $X$ is IH (see Remark \ref{rem:grad} for the converse). 
Similarly for any $1<p<\infty$, we can define $H^{1,p}_0(U), H^{1,p}_{\mathrm{loc}}(U)$ and $H^{1,p}(U)$, let us refer \cite{BjornBjorn, HKST} for the details.
\begin{remark}\label{rem:sobolev}
For a function $f:X \to \mathbb{R}$, we see that $f \in H^{1,2}(X)$ holds if and only if $f|_B \in H^{1,2}(B)$ holds for any ball $B$ with $|f|+|\nabla f| \in L^2(X)$. One implication is trivial, the other one is justified as follows.

Fix an open ball $B$ of radius $1$, and take Lipschitz cut-off functions $\rho_i$  on $X$ with $|\nabla \rho_i| \le \frac{1}{i}$, $\rho_i=1$ on $iB$, and compact supports in $i^2B$. Consider $\rho_if \in H^{1,2}(X)$, then the locality of the gradient (\ref{eq:localitygrad}) allows us to compute
\begin{equation}
    |\nabla (\rho_i f)| \le |\nabla \rho_i| |f| + \rho_i |\nabla f| \le \frac{1}{i}|f|+|\nabla f|, \quad \text{for $\meas_X$-a.e. on $X$.}
\end{equation}
Thus we see that $\{\rho_if_i\}_i$ is a bounded sequence in $H^{1,2}(X)$.
Taking $i \to \infty$ with Mazur's lemma proves $f \in H^{1,2}(X)$.
\end{remark}
Finally let us localize the Laplacian as follows (see also \cite[Proposition 6.4 and Lemma 6.6]{AmbrosioMondinoSavare}).
\begin{definition}[$D(\Delta, U), D_{\mathrm{loc}}(\Delta, U)$]
Assume that $X$ is IH. Then let us define two domains of \textit{local} Laplacian as follows;
\begin{enumerate}
    \item{($D_{\mathrm{loc}}(\Delta, U)$)} define $D_{\mathrm{loc}}(\Delta, U)$ by the set of all $f \in H^{1,2}_{\mathrm{loc}}(U)$ satisfying the following:
    \begin{itemize}
        \item there exists $h \in L^2_{\mathrm{loc}}(U)$ such that
    \begin{equation}
        \int_U\langle \nabla f, \nabla \phi \rangle \di \meas_X=-\int_Uh\phi \di \meas,\quad \text{for any $\phi \in \mathrm{Lip}_c(U)$.}
    \end{equation}
    \end{itemize}
    Since $h$ is unique, we denote it by $\Delta_U f$. Then since the locality of the Laplacian as (\ref{eq:localitylap}) also holds in this local case, we can use a simplified notation $\Delta f=\Delta_Uf$ with no confusion;
    \item{($D(\Delta, U)$)} define $D(\Delta, U)$ by the set of all $f \in D_{\mathrm{loc}}(\Delta, U)$ with $|f|+|\nabla f| +|\Delta f| \in L^2(U).$
\end{enumerate}
\end{definition}
Note that in the case when $X$ is an $\RCD(K, N)$ space, thanks to the locality of gradient (\ref{eq:localitygrad}), by multiplying good cut-off functions, for any $f \in D(\Delta, U)$, the Hessian, $\mathrm{Hess}_f$, is well-defined as an element of $L^2_{\mathrm{loc}}$-tensor of $(0,2)$-type on $U$. This will play a role later.
\subsection{PI condition}\label{sub:pi}
Throughout this subsection, let us fix a metric measure space $X$.
\begin{definition}[PI condition]\label{def:PI}
We say that;
\begin{enumerate}
    \item{(Local volume doubling)} $X$ is \textit{local volume doubling} for some function $c_v:(0, \infty) \to [1, \infty)$ if 
\begin{equation}
    \meas_X(2B) \le c_v(R)\meas_X(B)
\end{equation}
holds for any open ball $B$ of radius at most $R$.
\item{(Local Poincar\'e inequality)} $X$ satisfies the \textit{local ($(1,2)$-)Poincar\'e inequality} for some function $c_p:(0, \infty) \to (0, \infty)$ and some $\Lambda\ge 1$ if for all $r \le R$, open ball $B$ of radius $r$ and Lipschitz function $f$ on $X$, we have
\begin{equation}
    \intav_B \left| f-\intav_Bf\di \meas_X\right|\di \meas_X \le c_p(R)r\left( \intav_{\Lambda B}(\mathrm{Lip}f)^2\di \meas_X\right)^{\frac{1}{2}}.
\end{equation}
\item{(PI space)} $X$ is \textit{PI} if $X$ is local volume doubling and it satisfies a local $(1,2)$-Poincar\'e inequality.
\end{enumerate}
\end{definition}
\begin{remark}\label{rem:bilipinv}
Of course there exist the corresponding \textit{global} notions as follows;
    we say that 
    \begin{enumerate}
        \item $X$ is \textit{global volume doubling} if it is locally volume doubling for a constant function $c_v\equiv c>1$;
        \item $X$ satisfies a \textit{global Poincar\'e inequality} if it satisfies the local $(1,2)$-Poincar\'e inequality for a constant function $c_p\equiv c>1$.
        \item $X$ is \textit{globally} PI if the above two conditions are satisfied.
    \end{enumerate}
    It is worth mentioning that
for all $K \in \mathbb{R}$ and $N<\infty$, any $\RCD(0, N)$ space and any compact $\RCD(K, N)$ space are globally PI. The global PI condition is a metric measure bi-Lipschitz invariant, see for instance \cite[Proposition 4.16]{BjornBjorn} (see also a recent survey \cite{Cap} about characterizations of PI conditions). 
\end{remark}
In order to introduce a couple of fundamental results on PI, let us recall basics in metric geometry. A metric space $Y$ is said to be a \textit{length space} if we have
\begin{equation}
    \dist_Y(x,y)=\inf_{\gamma}l(\gamma),\quad \text{for all $x, y \in Y$,}
\end{equation}
where the infimum in the right-hand-side runs over all curves $\gamma$ from $x$ to $y$, and $l(\gamma)$ stands for the length. Moreover we say that $Y$ is a \textit{geodesic space} if for all $x, y \in Y$, a curve from $x$ to $y$ attains the infimum above, in this case, after taking a re-parametrization, we can assume that $\gamma$ is an isometric embedding from the interval $[0, \dist_Y(x, y)]$ to $Y$, then it is called a \textit{(minimal) geodesic}. It is well-known that if $Y$ is \textit{proper} (namely any closed ball is compact) and it is a length space, then it is a geodesic space. See for instance \cite[Corollary 2.5.20]{BBI}. It is worth mentioning that, thanks to the next proposition, eventually we deal with only proper geodesic spaces in our main results.
\begin{proposition}\label{prop:poincaresobolev}
Let us assume that $X$ is a length space and let $B$ be an open ball of radius $1$. Then we have the following.
\begin{enumerate}
    \item{(Proper and geodesic)} If $X$ is locally volume doubling, then we know that $X$ is proper, thus $X$ is a geodesic space. 
    \item{(Exponential volume growth)} If $X$ is locally volume doubling, then
    \begin{equation}\label{eq:expgrowth}
        \meas_X(rB) \le Ce^{Cr}\meas_X(B),\quad \text{for any $r>1$,}
    \end{equation}
    where $C>1$ depends only on $c_v(2)$. 
    \item{(Weak Ahlfors $Q$-regularity)}  If $X$ is locally volume doubling, then for any $R>0$ there exist $c \in (0, 1)$ and  $Q>2$ depending only on $c_v(R)$ such that for all $s<t \le R$ we have
    \begin{equation}\label{eq:ahlfors}
    \frac{\meas_X(sB)}{\meas_X(tB)} \ge c\left(\frac{s}{t}\right)^Q.
    \end{equation}
    \item{(Self-improvement on Poincar\'e-Sobolev inequality with $\Lambda=1$)} 
If $X$ is PI, then we have a local $(p, 2)$-Poincar\'e-Sobolev inequality; for all $f \in \mathrm{Lip}(X)$ and $r\le R$,
\begin{equation}
    \left(\intav_{rB} \left| f-\intav_{rB}f\di \meas_X\right|^p\di \meas_X\right)^{\frac{1}{p}} \le Cr\left( \intav_{rB}(\mathrm{Lip}f)^2\di \meas_X\right)^{\frac{1}{2}},
\end{equation}
where $p=\frac{2Q}{Q-2}>2$, $Q$ is obtained as above and $C>1$ depends only on $c_v(R), c_p(R)$ and $\Lambda$.
    \item{(Dirichlet Poincar\'e inequality)} If $X$ is PI with $X \setminus (1+\epsilon)RB \neq \emptyset$ for some $\epsilon>0$ and some $R>0$, then for any $r \le R$, we have
    \begin{equation}\label{eq:sobolevpoincare}
    \left(\intav_{rB}|f|^p\di \meas_X\right)^{\frac{1}{p}} \le C r \left(\intav_{rB}|\nabla f|^2\di \meas_X\right)^{\frac{1}{2}} \quad \text{for any $f \in H^{1,2}_0(rB)$,} 
\end{equation}
where $p>2$ as above, and $C>1$ depends only on $c_v(R), c_p(R), \Lambda, R$ and $\epsilon$.
In particular the canonical embedding from $H^{1,2}_0(rB)$ into $L^2(rB)$ is a compact operator.
    \item{(Minimal relaxed slope $=$ local slope)} If $X$ is PI, then for any $f \in H^{1,2}(X)$ and any $g \in \mathrm{Lip}(X)$, we have 
    \begin{equation}
        |\nabla f|(x)=\mathrm{Lip} g(x),\quad \text{for $\meas_X$-a.e. $x \in \{f=g\}$.}
    \end{equation}
\end{enumerate}
\end{proposition}
\begin{proof}
Let us refer (see also \cite{HK, HKST, KLV}); 
\begin{itemize}
    \item \cite[Theorem 2.5.28, Lemma 3.3, Theorem 4.21 and Corollary 4.40]{BjornBjorn} for (1), (3) and (4);
    \item \cite[(4.5) and Theorem 5.1]{Cheeger} (together with (\ref{eq:localitygrad})) for (5) and (6).
\end{itemize} 
    We here give only the proof of (\ref{eq:expgrowth}) because the techniques will be useful, see Remarks \ref{rem:covering} and \ref{rem:ramified}. 
The idea comes from \cite[Lemma 3.3]{CM}.
Fix a sufficiently large $r>1$.%

Firstly let us prove 
\begin{equation}\label{eq:inductive}
\meas_X\left(rB\setminus (r-1)B\right) \le C\meas_X\left((r-1)B\setminus (r-2)B\right),
\end{equation}
where $C>1$ depends only on $c_v(2)$.

The proof is as follows. Take $\{x_i\}_{i=1}^k$ be a maximal $1$-separated subset of $rB\setminus (r-1)B$. In particular
\begin{equation}
rB\setminus (r-1)B\subset \bigcup_{i=1}^kB_2(x_i).
\end{equation}
On the other hand,  for any $x_i$, take a minimal geodesic $\gamma_i:[0, \dist_X(x, x_i)] \to X$ from the center $x$ of $B$ to $x_i$, and put 
\begin{equation}
y_i:=\gamma_i\left(r-\frac{3}{2}\right) \in (r-1)B\setminus (r-2)B.
\end{equation}
Note that $\{B_{\frac{1}{10}}(y_i)\}_i$ is pairwise disjoint, included in $(r-1)B\setminus (r-2)B$.
Thus
\begin{align}
\meas_X\left(rB\setminus (r-1)B\right) &\le \sum_{i=1}^k\meas_X(B_2(x_i)) \nonumber \\
&\le C\sum_{i=1}^k\meas_X(B_{\frac{1}{10}}(y_i)) \le C\meas_X\left( (r-1)B\setminus (r-2)B\right)
\end{align}
which completes the proof of (\ref{eq:inductive}).

Similarly, denoting by
\begin{equation}
a_l:=\meas_X\left( (r-l)B\setminus (r-l-1)B\right),
\end{equation}
we have
\begin{equation}
a_l \le Ca_{l+1}.
\end{equation}
In particular 
\begin{equation}
a_l\le C^{r}\meas_X(B).
\end{equation}
Thus taking the sum with respect to $l$ completes the proof of (\ref{eq:expgrowth}).
\end{proof}
\begin{remark}
    The  PI condition does not imply a convexity of the regular set, for instance, the glued space of three segments by a common end point is globally PI, however the regular set, defined by the set of interior points of segments, is not connected.
\end{remark}
\begin{remark}\label{rem:localPI}
    In connection with (the proof of) (\ref{eq:expgrowth}), it is worth mentioning that, in order to justify our main results, local Poincar\'e inequalities on large balls are not necessary, namely, it is enough to consider balls of radius at most $1$ (or more generally at most a fixed radius).
\end{remark}
Recalling the Bishop-Gromov inequality in \cite[Theorem 5.31]{LottVillani}, in \cite[Theorem 
 2.3]{Sturm06b} independently, and a $(1, 1)$-Poincar\'e inequality in \cite[Theorem 1]{Rajala}, we have the following.
\begin{theorem}[PI for RCD]\label{thm:lpircd}
    Any $\RCD(K, N)$ space for some $K \in \mathbb{R}$ and some $N \in [1, \infty)$ is a PI geodesic space, where $c_v(R), c_p(R)$ can be chosen by depending only on $K, N$ and $R$.
\end{theorem}
\begin{remark}[PI implies Lipschitz representative]\label{liprepresentative}
    If $X$ is a PI space, then any $f \in H^{1,2}_{\mathrm{loc}}(X)$ has a Lipschitz representative. This is a direct consequence of a telescopic argument based on the Poincar\'e inequality, see for instance \cite[Section 4]{Cheeger}.
\end{remark}
Finally let us focus on Gaussian estimates of the heat kernel $p$ with respect to the time derivatives, as follows. 
\begin{theorem}[Gaussian bound on the heat kernel]\label{thm:gaus}
Assume that $X$ is a PI geodesic space with IH. Then the heat kernel $p$ is H\"older continuous. 
     Furthermore, for all $n \ge 0$, $t \le 1$ and $x, y \in X$, we have
     \begin{equation}\label{eqn:heat kernel bounds}
         \left|\partial_t^np(x, y, t)\right| \le \frac{C}{t^n\meas_X(B_{\sqrt{t}}(x))}\exp \left( -\frac{\dist_X(x,y)^2}{Ct}\right), 
     \end{equation}
     where $C>1$ depends only on $c_v(2), c_p(2)$ and $\Lambda$.
\end{theorem}
\begin{proof}
Firstly note that the H\"older continuity is the standard consequence of the parabolic Harnack inequality, see for instance \cite[Proposition 3.1]{Sturm96}. 
    Thanks to \cite[Theorem 2.6]{Sturm95} and \cite[Proposition 3.1]{Sturm96}, we know
    \begin{equation}
        p(x, y, t) \le \frac{C}{\sqrt{\meas_X(B_{\sqrt{t}}(x)\meas_X(B_{\sqrt{t}}(y))}} \exp \left( -\frac{\dist_X(x,y)^2}{Ct}\right). 
    \end{equation}
    Then the same arguments after \cite[Lemma 2.1]{Jiang15} allow us to conclude.
\end{proof}
\begin{remark}
    A lower bound on $p$ can be also obtained, see \cite[Theorem 4.8]{Sturm96} and discussions after \cite[Lemma 2.1]{Jiang15}, though we will not use it in the sequel. The upper bound $1$ on the time $t$ can be replaced by an arbitrary $T \ge 1$ (then $C$ also depends on $T$).
\end{remark}
\subsection{Quantitative Lipschitz continuity of harmonic functions (QL)}
Let us fix an IH metric measure space $X$. We start this subsection by introducing  the harmonicity of a function. 
\begin{definition}[Harmonic function]
We say that a function $f$ on an open subset $U$ of $X$ is said to be \textit{harmonic} if $f \in D_{\mathrm{loc}}(\Delta, U)$ with $\Delta f=0$.
\end{definition}
The following notion introduced in \cite[Definition 1.1]{JKY} plays a key role in the paper.
\begin{definition}[Quantitative Lipschitz continuity of harmonic functions (QL)]\label{def:qLCh}
We say that $X$ satisfies the \textit{local quantitative Lipschitz continuity of harmonic functions}, written by QL for short in the sequel, for some non-decreasing function $c_h:(0, \infty) \to [1, \infty)$ if $f:B \to \mathbb{R}$ is a harmonic function on an open ball $B$ of radius $r$, then
\begin{equation}
    \sup_{\frac{1}{2}B}|\nabla f| \le \frac{c_h(r)}{r} \intav_B|f|\di \meas_X.
\end{equation}
\end{definition}
\begin{remark}\label{rem:qlch}
    It is worth mentioning that for our purposes, it is enough to consider $c_h(r)$  for any $r \le 1$ because other case, say $r>1$, is justified by the case when $r=1$ because of taking a maximal $1$-separated subset on a given ball of radius $r>1$, together with the local volume doubling condition.
\end{remark}
\begin{remark}\label{rem:positive}
    Under assuming that $X$ is PI, if we get a quantitative Cheng-Yau type gradient estimate for any positive harmonic function $f$ on an open ball $B$ of radius $r>0$:
    \begin{equation}\label{eq:positiveharm}
        \frac{|\nabla f|}{f} \le \frac{C}{r},\quad \text{on $\frac{1}{2}B$,}
    \end{equation}
    then we also have a QL because of the following, due to a pointing out around  \cite[(1.6)]{JKY}.

    Take a harmonic function $f$ on an open ball $B$ of radius $r$, and consider an open ball $\tilde B$ of radius $\frac{r}{8}$ centered at a point in $\frac{1}{2}B$. Then applying (\ref{eq:positiveharm}) to $f|_{\tilde B}-\inf_{\tilde B}f$, we have on $\tilde B$
    \begin{equation}
        |\nabla f| \le \frac{C}{r}(f-\inf_{\tilde B}f) \le \frac{C}{r}\|f\|_{L^{\infty}(\tilde B)}\le \frac{C}{r}\intav_{2\tilde B}|f|\di \meas_X
    \end{equation}
    which proves the desired QL,
    where we used a quantitative $L^{\infty}$-estimate for harmonic functions in the last inequality, see \cite[Lemma 2.1]{JKY}.
\end{remark}
Let us introduce the following result about a QL for a $\RCD$ space proved in \cite[Theorems 1,1 and 1.2]{Jiang}.
\begin{theorem}[QL for RCD]\label{thm:qLCh}
    Any $\RCD(K, N)$ space for some $K \in \mathbb{R}$ and some $N \in [1, \infty)$ satisfies the QL for some $c_{K, N}:(0,\infty) \to (0, \infty)$ depending only on $K, N$.
\end{theorem}
The following is due to \cite[Theorem 3.1 (or Proposition 3.2)]{JKY}.
\begin{theorem}[Lipschitz regularity of solution of Poisson equation]\label{thm:Poisson}
Assume that $X$ is a PI geodesic space with a QL.
Let $0<r \le R$ and let $B$ be an open ball of radius $r$ and let $f \in D(\Delta, B)$ with $\Delta f \in L^{\infty}(B)$. Then for $\meas_X$-a.e. $y \in \frac{1}{2}B$ we have for $q=\max\{\frac{3}{2}, \frac{2Q}{Q+2}\}$ (recall $Q$ as in (\ref{eq:ahlfors}))
\begin{equation}
    |\nabla f|(y) \le C \left( \frac{1}{r}\intav_B|f|\di \meas_X+\sum_{j=-\infty}^{\lfloor \log_2\frac{R}{2}\rfloor}2^j\left( \intav_{B_{2^j}(y)}|\Delta f|^q\right)^{\frac{1}{q}}\right),
\end{equation}
where $C$ depends only on $c_v(2R), c_p(2R)$ and $c_h(2R)$.

\end{theorem}
Then we can prove a Gaussian gradient estimate on the heat kernel $p$ of $X$ (see also \cite{Jiang15, JiangLiZhang} for $\RCD$ spaces).
\begin{theorem}[Gradient estimate on the heat kernel]\label{thm:gradient}
If $X$ is a PI geodesic space with a QL, then the heat kernel $p$ of $X$ satisfies for $t\le 1$
\begin{equation}
    |\nabla_xp(x, y, t)|\le \frac{C}{\sqrt{t}\meas_X(B_{\sqrt{t}}(x))}\exp \left(-\frac{\dist_X(x, y)^2}{Ct}\right),
\end{equation}
where $C>1$ depends only on $c_v(2), c_p(2), \Lambda$ and $c_h(1)$.
\end{theorem}
\begin{proof}
With no change of the proof of \cite[Theorem 3.2]{Jiang15} together with (\ref{eqn:heat kernel bounds}), we can realize this. 
\end{proof}
In order to provide applications of the result above, let us prepare the following
which is justified by the same arguments as in the proof of \cite[Lemma 2.7]{BGHZ} together with (\ref{eq:expgrowth}), where our current assumption, IH for $X$, does not play a role in the proof. 
\begin{lemma}\label{lem:integralgaussian}
Assume that $X$ is a geodesic space with a local volume doubling condition. 
Then for all $t \in (0, 1]$, $\alpha \in \mathbb{R}$ and $\beta \in (0, \infty)$ we have
\begin{equation}
    \int_X\meas_X(B_{\sqrt{t}}(x))^{\alpha}\exp\left(-\frac{\beta\dist_X(x,y)^2}{t}\right)\di \meas_X(x) \le C\meas_X(B_{\sqrt{t}}(y))^{\alpha+1},
\end{equation}
where $C>0$ depends only on $c_v(1), \alpha$ and $\beta$.
\end{lemma}
The following will play a key role to give a characterization of $\RCD$ spaces, see Theorem \ref{thm:fromalmostrcdtorcd}.
\begin{corollary}\label{cor:lipreg}
    Assume that $X$ is a PI geodesic space with a QL. Then for all $f \in L^{\infty}\cap L^{2}(X)$ and $t>0$, we know $|\nabla h_tf| \in L^{\infty}(X)$.
\end{corollary}
\begin{proof}
    Since for all $\phi, \psi \in \mathrm{Lip}_c(X)$ we have
    \begin{equation}
        \int_X\psi \langle \nabla h_tf, \nabla \phi \rangle \di \meas_X =\int_X\int_X\langle \nabla_xp(x, y, t), \nabla \phi(x)\rangle \psi(x) f(y)\di \meas_X(x) \di \meas_X(y),
    \end{equation}
    we conclude by Lemma \ref{lem:integralgaussian} (see also \cite[subsection 3.1]{BGHZ}) that for $\meas_X$-a.e. $x \in X$
    \begin{align}
        |\nabla h_tf|(x) &\le C \int_X|\nabla_xp|\di \meas_X(y) \nonumber \\
        &\le \int_X \frac{C}{\sqrt{t}\meas_X(B_{\sqrt{t}}(x))}\exp \left(-\frac{\dist_X(x, y)^2}{Ct}\right) \di \meas_X \le \frac{C}{\sqrt{t}}.
    \end{align}
\end{proof}
\begin{theorem}[Existence of good cut-off function]\label{thm:cutoff}
Assume that $X$ is a PI geodesic space with a QL. Then for all compact subset $A$ and open neighborhood $U$ of $A$, there exists $\rho \in D(\Delta)$, called a good cut-off function in the sequel, such that $0 \le \rho \le 1$, that $\rho=1$ on $A$, that $\supp \rho \subset U$ and that
\begin{equation}
    |\nabla \rho| + |\Delta \rho| \in L^{\infty}(X).
\end{equation}
Moreover for all open ball $B$ of radius $1$ and $R \ge 1$, there exists $\phi \in D(\Delta)$ such that $0\le \phi \le 1$, $\phi|_{RB}=1$, $\supp \phi \subset (R+2)B$, and 
\begin{equation}
    |\nabla \phi| + |\Delta \phi| \le C, 
    \end{equation}
    where $C>1$ depends only on $c_v(4), c_p(2), \Lambda$ and $c_h(1)$.
\end{theorem}
\begin{proof}
    For the first statement about $\rho$, by applying the standard construction of a partition of unity, 
    it is enough to discuss only the case when $A$ is a closed ball $\bar B$ and $U=(1+\epsilon)B$ for some $\epsilon>0$.

    Let us follow the same strategy of the proof of \cite[Lemma 3.1]{MN}. 
    Take a cut-off function $\psi \in \mathrm{Lip}_c(X)$ with $\psi|_B=1$, $\supp \psi \subset (1+\epsilon)B$ and 
    \begin{equation}
        |\nabla \psi| \le \frac{1}{\epsilon}.
    \end{equation}
    Note that for any $t \le 1$, Theorems \ref{thm:gaus} and \ref{thm:gradient} yield
    \begin{align}\label{eq:laplacianbound}
        |\Delta h_t\psi(y)| \nonumber &=\left| \int_X\Delta_yp(x, y, t)\psi(x)\di \meas_X(x)\right| \nonumber \\
        &=\left| \int_X\Delta_xp(x, y, t)\psi(x)\di \meas_X(x)\right| \nonumber \\
        &=\left| \int_X\langle \nabla_xp(x, y, t), \nabla \psi(x)\rangle \di \meas_X(x)\right| \nonumber \\
        &\le \int_B\frac{C}{\epsilon \sqrt{t}\meas_X(B_{\sqrt{t}}(x))}\exp \left(-\frac{\dist_X(x,y)^2}{ct}\right) \di \meas_X(x) \le \frac{C}{\epsilon \sqrt{t}},
    \end{align}
    where we used Lemma \ref{lem:integralgaussian} in the last inequality. Thus
    \begin{equation}\label{eq:difference}
        |h_t\psi(x)-\psi(x)| \le C\int_0^t|\Delta h_s\psi(y)|\di s \le \frac{C}{\epsilon}\int_0^t\frac{1}{\sqrt{s}}\di s\le \frac{C\sqrt{t}}{\epsilon}.
    \end{equation}
    Thus there exists $t_0 \in (0, 1]$ such that $h_{t_{0}}\psi(y) \in [\frac{3}{4}, 1]$ holds for any $y \in B$ and that $h_{t_0}\psi(y) \in [0, \frac{1}{4}]$ holds for any $y \not \in (1+\epsilon)B$ (note that $t_0$ can be taken quantitatively).
    Then finding a smooth function $f:[0,1] \to [0,1]$ satisfying that $f=1$ on $[\frac{3}{4}, 1]$ and that $f=0$ on $[0, \frac{1}{4}]$, the function $\rho:=f(h_{t_0}\psi)$ provides a desired cut-off function, where the desired gradient estimate comes from Theorem \ref{thm:Poisson}, and we immediately used 
    a formula (see for instance \cite[Proposition 5.2.3]{GP}):
    \begin{equation}\label{eq:formula}
        \Delta (f \circ \phi)=\frac{d^2 f}{d x^2}(\phi)|\nabla \phi|^2+\frac{d f}{d x}(\phi) \Delta \phi 
    \end{equation}
for all $f \in C^{1,1}(\mathbb{R})$ and $\phi \in D(\Delta)$. 

    In order to prove the second statement, take a maximal $1$-separated subset $\{x_i\}_i$ of $X$, consider the open ball $B_i$ of radius $1$ centered at $x_i$, and then take a good cut-off $\rho_i \in D(\Delta)$ with $\rho_i=1$ on $B_i$ and $\supp \rho_i\subset 2B_i$ constructed above (as $\epsilon=1$). Note that for any $i$, we have
    \begin{equation}\label{eq:number4}
        \sharp \left\{j ; 2B_i \cap 2B_j \neq \emptyset \right\} \le C,
    \end{equation}
where $C>1$ depends only on $c_v(4)$. Let 
\begin{equation}
    \phi_i:=\frac{\rho_i}{\sum_j\rho_j}.
\end{equation}
    For a fixed ball $B$ of radius greater than $1$, consider 
    \begin{equation}
        \mathbb{N}_B:=\left\{i ; 2B_i \cap B \neq \emptyset\right\},
    \end{equation}
    and then define
    \begin{equation}
        \phi:=\sum_{i \in \mathbb{N}_B}\phi_i,
    \end{equation}
    which provides a desired function because of (\ref{eq:number4}).
    \end{proof}

    \begin{remark}[Test function]\label{rem:test}
Let us assume that $X$ is a PI geodesic space with a QL.
Then, as in the case of $\RCD$ spaces, we can, of course, define test functions by (\ref{eq:testdef}).
It is proved in \cite[Proposition 3.1.3]{Gigli} that $\mathrm{Test}(X)$ is an algebra if $X$ is an $\RCD(K, \infty)$ space for some $K \in \mathbb{R}$.
   However, in this general setting, we do not know this property, more precisely, it is unclear whether $\langle \nabla f, \nabla g \rangle \in H^{1,2}(X)$ for $f, g \in \mathrm{Test}(X)$. We will need to take care about this later, for instance, in (the proof of) Proposition \ref{prop:charbe}. See also Question \ref{ques:1}.
\end{remark}
Though the following does not play an essential role in the paper, this has an independent interest, from the point of view of an existence of \textit{Laplacian cut-off functions}, see \cite[Definition 2.1]{Gun}.
\begin{proposition}[Laplacian cut-off function]\label{rem:lapcut}
    Assume that $X$ is an $\RCD(K, N)$ space for some $K\in \mathbb{R}$ and some $N \in [1, \infty)$, and let us take an open ball $B$ of radius $1$.
    Then there exist a sequence $R_i \uparrow \infty$ and a sequence of $\phi_i \in \mathrm{Test}_c(X)$ such that $0\le \phi_i \le 1$, that $\phi_i|_{R_iB}=1$, that $\supp \phi_i \subset R_{i+1}B$, and that
\begin{equation}
    \|\nabla \phi_i\|_{L^{\infty}}+ \|\Delta \phi_i\|_{L^{\infty}} \to 0.
\end{equation}
\end{proposition}
\begin{proof}
    For any $i \ge 1$, take a Lipschitz cut-off $\psi_i$ on $X$ with $\psi_i|_{iB}=1$, $\supp \psi_i \subset i^2B$ and $|\nabla \psi_i| \le \frac{1}{i}$. Then considering $h_t\psi_i$ with the same arguments as in the proof of Theorem \ref{thm:cutoff} by taking $t=1$ (then $|\Delta h_t\psi_i|$ is small), we get the conclusions, except for a result that $\phi_i=f(h_t\psi_i) \in \mathrm{Test}(X)$, where we immediately used the \textit{Bakry-\'Emery gradient estimate} (see \cite[Corollary 4.3]{Savare}):
\begin{equation}
    |\nabla h_t\psi_i| \le e^{-Kt}h_t|\nabla \psi_i| \le e^{-tK}\|\nabla \psi_i\|_{L^{\infty}} \le \frac{e^{-tK}}{i}.
\end{equation} 
Then, the remaining one, $\phi_i \in \mathrm{Test}(X)$, is an easy consequence of (\ref{eq:formula}) together with (\ref{eq:bine}).
\end{proof}
\section{Local lower bound on Ricci curvature}\label{sec:3}
Let $X$ be an IH metric measure space.
The purpose of this section is to discuss \textit{local} lower bounds on Ricci curvature in a synthetic sense, more precisely, from the point of view of the Bakry-\'Emery theory \cite{BE} again, which is different from the \textit{optimal transportation theory}, see for instance \cite{AmbrosioMondinoSavare, ErbarKuwadaSturm, LottVillani, Sturm06, Sturm06b}.
\subsection{Locally BE spaces}
Our working definition is the following, see \cite[Definition 2.1]{AmbrosioMondinoSavare2} in the case when $U=X$.
\begin{definition}[Locally BE space]\label{def:localrcd}
We say that an open subset $U$ of $X$ is a \textit{locally $\BE(K, N)$ space} for some $K \in \mathbb{R}$ and some $N \in [1, \infty]$ if
for any $f \in D(\Delta, U)$ with $\Delta f \in H^{1,2}(U)$,
\begin{equation}
\frac{1}{2}\int_U\Delta \phi |\nabla f|^2 \di \meas_X \ge \int_U\phi \left( \frac{(\Delta f)^2}{N}+\langle \nabla \Delta f, \nabla f\rangle +K|\nabla f|^2\right)\di \meas_X
\end{equation}
holds for any $\phi \in D(\Delta) \cap L^{\infty}(X)$ with compact support in $U$, $\phi \ge 0$ and $\Delta \phi \in L^{\infty}(X)$. 
\end{definition}
It is trivial that
\begin{itemize}
    \item if $X$ is a $\BE(K, N)$ space, then $X$ is a locally $\BE(K, N)$ space;
    \item in the case when $X$ is compact,  $X$ is a locally $\BE(K,N)$ space if and only if $X$ is a $\BE(K,N)$ space.
\end{itemize} The converse implication of the first statement above will be discussed in Proposition \ref{prop:LtoG} and Theorem \ref{thm:LtoG}. 

Thanks to the regularity result stated in Theorem \ref{thm:Poisson}, we provide  equivalent formulations of locally $\BE$ spaces, under PI with a QL, as follows.
\begin{proposition}\label{prop:charbe}
Assume that $X$ is a PI geodesic space with a QL.
Let $U$ be an open subset of $X$. Then for all $K \in \mathbb{R}$ and $N \in [1, \infty]$, the following two conditions are equivalent.
\begin{enumerate}
    \item $U$ is a locally $\BE(K, N)$ space.
    \item We have
\begin{align}\label{eq:weakboch}
    \frac{1}{2}\int_X\Delta \phi |\nabla f|^2\di \meas_X  \ge \int_X\left(\phi \frac{(\Delta f)^2}{N}-\phi (\Delta f)^2-\Delta f\langle \nabla f, \nabla \phi \rangle +K\phi|\nabla f|^2\right) \di \meas_X
\end{align}
for any $\phi \in D(\Delta) \cap L^{\infty}(X)$ with $\phi \ge 0$, compact support in $U$ and $\Delta \phi \in L^{\infty}(X)$, and for any $f \in D(\Delta)$ with compact support in $U$. 
\end{enumerate} 
\end{proposition}
\begin{proof}
Firstly let us prove that if (2) holds, then (1) holds.
Thus assume (2). Fix $f \in D(\Delta, U)$ with $\Delta f \in H^{1,2}(U)$, and fix $\phi \in D(\Delta) \cap L^{\infty}(X)$ with $\phi \ge 0$, compact support in $U$, and $\Delta \phi \in L^{\infty}(X)$.
Find a good cut-off function $\rho \in D(\Delta)$ with $\rho=1$ on a neighborhood of $\supp \phi$ included in $U$, and $\supp \rho \subset U$, constructed in Theorem \ref{thm:cutoff}. Note that $\rho f, \rho \phi \in D(\Delta)$ with $\Delta (\rho \phi) \in L^{\infty}(X)$. 

Then
\begin{align}\label{al:locboch}
    &\frac{1}{2}\int_U\Delta (\rho\phi) |\nabla (\rho f)|^2 \di \meas_X \nonumber \\
&\ge \int_U\left(\rho \phi \frac{(\Delta (\rho f))^2}{N}-\rho\phi (\Delta (\rho f))^2-\Delta (\rho f)\langle \nabla (\rho f), \nabla (\rho\phi) \rangle +K\rho\phi |\nabla (\rho f))|^2\right) \di \meas_X.
\end{align}
The left-hand-side of (\ref{al:locboch}) is, by (\ref{eq:localitygrad}) and (\ref{eq:localitylap}), equal to
\begin{equation}
    \frac{1}{2}\int_U\Delta \phi |\nabla f|^2\di \meas_X.
\end{equation}
Similarly, the right-hand-side of  (\ref{al:locboch}) is equal to
\begin{align}
    &\int_U\left(\phi \frac{(\Delta f)^2}{N}-\phi (\Delta f)^2-\Delta f\langle \nabla f, \nabla \phi \rangle +K\phi |\nabla f|^2\right) \di \meas_X \nonumber \\
    &=\int_U\phi \left( \frac{(\Delta f)^2}{N}+\langle \nabla \Delta f, \nabla f\rangle +K|\nabla f|^2\right)\di \meas_X.
\end{align}
Thus we have (1).

Next let us assume (1). Take $\phi \in D(\Delta) \cap L^{\infty}(X)$ with $\phi \ge 0$, compact support in $U$ and $\Delta \phi \in L^{\infty}(X)$,  fix $f \in D(\Delta, U)$, and take $\rho$ as above again. Considering $h_t(\rho f)$, 
we know by (1)
\begin{align}
&\frac{1}{2}\int_U\Delta \phi |\nabla (h_t (\rho f))|^2 \di \meas_X \nonumber \\
&\ge \int_U\phi \left(\frac{(\Delta ( h_t(\rho f)))^2}{N}+\langle \nabla \Delta (h_t(\rho f)), \nabla (h_t (\rho f))\rangle +K|\nabla (h_t(\rho f))|^2\right)\di \meas_X \nonumber \\
&= \int_U\left(\phi \frac{(\Delta (h_t(\rho f)))^2}{N}-\phi (\Delta (h_t(\rho f)))^2-\Delta (h_t(\rho f))\langle \nabla (h_t(\rho f)), \nabla \phi \rangle +K\phi |\nabla (h_t(\rho f))|^2\right) \di \meas_X.
\end{align}
Thus letting $t \to 0$ completes the proof of (2) because of applying (\ref{eq:localitygrad}) and (\ref{eq:localitylap}) again.
\end{proof}
\begin{corollary}[Global-to-Local]\label{cor:gtol}
Assume that $X$ is a PI geodesic space with a QL.  Then
for all open subsets $U\subset V $ of $X$, if $V$ is a locally $\BE(K, N)$ space, then so is $U$.
\end{corollary}
\begin{proof}
    Consider (2) of Proposition \ref{prop:charbe}.
\end{proof}
This should be compared with a Global-to-Local result \cite[Proposition 7.7]{AmbrosioMondinoSavare2}, stating that if $X$ is an $\RCD(K, N)$ space, $U$ is an open subset of $X$, and $(U, \dist_X)$ is geodesic (thus it is convex in $X$) with $\meas_X(\partial U)=0$, then the closure $\bar U$ is an $\RCD(K, N)$ space with the restriction of $\meas_X$ to $\bar U$.

Using a partition of unity by good cut-off functions constructed in Theorem \ref{thm:cutoff}, together with Theorem \ref{thm:Poisson} and Proposition \ref{prop:charbe}, we can easily get the following.
\begin{corollary}\label{prop:LtoG}
Assume that $X$ is a PI geodesic space with a QL.
    Let $\{U_i\}_i$ be a family of open subsets of locally $\BE(K, N)$ spaces in $X$. 
    Then
    \begin{equation}
        U:=\bigcup_iU_i
    \end{equation}
    is also a locally $\BE(K, N)$ space.
\end{corollary}

As a reverse direction of Corollary \ref{cor:gtol}, the following is non-trivial.
\begin{theorem}[Local-to-Global]\label{thm:LtoG}
Assume that $X$ is a PI geodesic space with a QL.
If $X$ is a locally $\BE(K, N)$ space, then $X$ is a $\BE(K, N)$ space. In particular, if we furthermore assume that $X$ verifies the SL, then $X$ is an $\RCD(K, N)$ space. 
\end{theorem}
\begin{proof}
    Fix an open ball $B$ of radius $1$. Applying Theorem \ref{thm:cutoff} (see also Proposition \ref{rem:lapcut}), find a sequence of cut-off functions $\rho_i \in D(\Delta)$ with $\rho_i|_{iB}=1$, $\supp \rho_i \subset (i+2)B$ and
    \begin{equation}\label{eq:cutoffbound}
        \sup_i\left( \|\nabla \rho_i\|_{L^{\infty}} + \|\Delta \rho_i \|_{L^{\infty}}\right)<\infty.
    \end{equation}
    Then for any $f \in D(\Delta)$ with $\Delta f \in H^{1,2}(X)$, and any $\phi \in D(\Delta) \cap L^{\infty}(X)$ with $\phi \ge 0$ and $\Delta \phi \in L^{\infty}(X)$, since $\phi \rho_i \in D(\Delta)$ with $\Delta (\phi \rho_i) \in L^{\infty}(X)$ because of applying Theorem \ref{thm:Poisson} to $\phi$, we have by definition
    \begin{equation}
        \frac{1}{2}\int_X\Delta (\phi\rho_i) |\nabla f|^2 \di \meas_X \ge \int_X\phi\rho_i \left( \frac{(\Delta f)^2}{N}+\langle \nabla \Delta f, \nabla f\rangle +K|\nabla f|^2\right)\di \meas_X. 
    \end{equation}
    Thus taking $i \to \infty$ in the above together with 
    (\ref{eq:cutoffbound}) and the dominated convergence theorem proves
    \begin{equation}
        \frac{1}{2}\int_X\Delta \phi |\nabla f|^2 \di \meas_X \ge \int_X\phi \left( \frac{(\Delta f)^2}{N}+\langle \nabla \Delta f, \nabla f\rangle +K|\nabla f|^2\right)\di \meas_X
    \end{equation}
    which shows that $X$ is a $\BE(K, N)$ space. Therefore the desired conclusion comes from this together with 
    (\ref{eq:expgrowth}).
    
\end{proof}
\begin{remark}\label{rem:covering}
    Let $X$ be an $\RCD(K, N)$ space for some $K \in \mathbb{R}$ and some $N \in [1, \infty)$, let $\pi: \tilde X \to X$ be a connected covering space and consider the lifted length distance $\dist_{\tilde X}$ and the lifted Borel measure $\meas_{\tilde X}$ on $\tilde X$ (see subsection 2.2 of \cite{MW} for the details). Though it follows from \cite[Theorem 1.1]{MW} after \cite{SW1, SW2} (see also \cite[Corollary 1.1]{W} and \cite[Theorem 18]{Z}) that $\tilde X$ is also an $\RCD(K, N)$ space, let us provide an alternative proof of it, based on our framework, in the case when $X$ is compact as follows.

    Firstly notice that for any $x \in X$ there exists $r_x>0$ such that for any $\tilde x \in \pi^{-1}(x)$, $\pi$ gives an isometry from $B_{r_x}(\tilde x)$ to $B_{r_x}(x)$ as metric measure spaces. Recalling the compactness of $X$, we can assume that there exists $r>0$ such that $r<r_x$ holds for any $x \in X$. Taking a finite covering $\{B_{\frac{r}{4}}(x_i)\}_i$ of $X$ and considering their lifts, we can conclude that $\tilde X$ is local volume doubling with a Poincar\'e inequality and a QL, at most scale $\frac{r}{8}$. 

On the other hand, it is easy to see that $\tilde X$ is complete and locally compact. Thus we know by the Hopf-Rinow theorem for metric spaces (see for instance \cite[Theorem 2.5.28]{BBI}) that $\tilde X$ is proper because it is a length space. In particular $\tilde X$ is a separable geodesic space. 

For any $f \in H^{1,2}(\tilde X)$ with $|\nabla f| \le 1$, the Sobolev-to-Lipschitz property on $X$ together with the observation above yields that $f$ has a continuous representative, in fact,
\begin{equation}
    \bar f(x):= \lim_{s\to 0}\intav_{B_s(x)}f\di \meas_{\tilde X}
\end{equation}
gives the desired representative. Then the same arguments as in the proof of Lemma \ref{prop:PIStoL} proves that $\bar f$ $1$-Lipschitz, namely we have the Sobolev-to-Lipschitz property on $\tilde X$.

    The above observations together with Remarks \ref{rem:localPI}, \ref{rem:qlch} and (the proof of) Theorem \ref{thm:LtoG} allow us to conclude that $\tilde X$ is an $\RCD(K, N)$ space.
\end{remark}
\subsection{Locally RCD space}
Finally, as in the locally BE case, let us localize the RCD theory as follows.
\begin{definition}[Locally RCD space]\label{def:lrcd}
We say that an open subset $U$ of a metric measure space $X$ is a \textit{locally $\RCD(K, N)$ space} for some $K \in \mathbb{R}$ and some $N \in [1, \infty]$ if for any $x \in U$ there exist an open subset $V$ of an $\RCD(K, N)$ space $Y$ and an open neighborhood $W$ of $x$ such that $W$ is isometric to $V$ as metric measure spaces. Moreover $X$ is said to be \textit{non-collapsed} if the above $Y$ can be chosen as a non-collapsed $\RCD(K, N)$ space.
\end{definition}
The following is straightforward by definition.
\begin{proposition}\label{prop:belo}
    A locally $\RCD(K, N)$ open subset for some $K \in \mathbb{R}$ and some $N \in [1, \infty)$ of an IH metric measure space $X$ is a locally $\BE(K, N)$ space.
\end{proposition}
\begin{proof}
The proof is very similar to that of Theorem \ref{thm:LtoG}.

Let us take an open subset $U$ of $X$ which is a locally $\RCD(K, N)$ space in $X$. Fix  $f \in D(\Delta, U)$ with $\Delta f \in H^{1,2}(U)$, and take $\phi \in D(\Delta)\cap L^{\infty}(X)$ with compact support in $U$, $\phi \ge 0$ and $\Delta \phi \in L^{\infty}(X)$ (note $|\nabla \phi| \in L^{\infty}(X)$). Find an open covering $\{B_{r_i}(x_i)\}_{i=1}^k$ of $\supp \phi$
satisfying that each $B_{3r_i}(x_i)$ is included in $U$ and it is isometric to a ball of radius $3r_i$ in an $\RCD(K, N)$ space as metric measure spaces. Take a good cut-off function $\psi_i \in D(\Delta)$ with $\psi_i=1$ on $B_{r_i}(x_i)$, $\supp \psi_i \subset B_{2r_i}(x_i)$, and $|\nabla \psi_i| +|\Delta \psi_i| \in L^{\infty}(X)$. Let 
\begin{equation}
    \phi_i:=\frac{\psi_i}{\sum_j\psi_j}.
\end{equation}
Since for each $i$
\begin{equation}
    \frac{1}{2}\int_X\Delta (\phi \phi_i)|\nabla f|^2 \di \meas_X \ge \int_X\phi \phi_i \left(\frac{(\Delta f)^2}{N}+\langle \nabla \Delta f, \nabla f\rangle +K|\nabla f|^2\right)\di \meas_X,
\end{equation}
taking the sum with respect to $i$ proves
\begin{equation}
    \frac{1}{2}\int_X\Delta \phi|\nabla f|^2 \di \meas_X \ge \int_X\phi \left(\frac{(\Delta f)^2}{N}+\langle \nabla \Delta f, \nabla f\rangle +K|\nabla f|^2\right)\di \meas_X,
\end{equation}
which shows that $U$ is a locally $\BE(K, N)$ space.
\end{proof}
Note that the converse implication is not true even in the compact case by considering the glued space of two spheres at a common point, as observed in \cite[Example 3.8]{Honda3} (see also Remark \ref{rem:example}). 

This local notion will play a role in Section \ref{sec:almostrcd}. It is important that by the RCD theory, we can use the Hessians of certain functions on an locally $\RCD$ open subset, as explained in Theorem \ref{thm:hessboch} and around the end of subsection \ref{subsec:localobj}. We do not know whether it is also possible for locally BE open subsets. See also Question \ref{ques:1}.

\section{Almost smooth metric measure space}\label{sec:as}
In this section we discuss our main targets, called \textit{almost smooth} metric measure spaces. 
\subsection{Smooth objects}
Let us recall fundamental smooth objects we will discuss later.
\begin{definition}\label{def:weighted}
Let $M^n=(M^n, g)$ be a (not necessarily complete) Riemannian manifold of dimension $n$ without boundary, and let $w \in C^{\infty}(M^n)$. 
\begin{enumerate}
\item{(Weighted volume)} For any Borel subset $A$ of $M^n$, let us define the \textit{weighted (Riemannian) volume measure} of $A$ associated with $w$ by
\begin{equation}
\mathrm{vol}^g_w(A):=\int_Ae^{-w}\di \mathrm{vol}^g,
\end{equation}
where $\mathrm{vol}^g$ denotes the Riemannian volume measure associated with $g$ (thus this coincides with the Hausdorff measure $\haus^n$ of dimension $n$).
\item{(Weighted Laplacian)} For any $\phi \in C^{\infty}(M^n)$, let us define the \textit{weighted Laplacian} associated with $w$ by
\begin{equation}
\Delta_w^g\phi:=\Delta^g\phi-g(\nabla^g \phi, \nabla^g w),
\end{equation}
where $\Delta^g\phi$ is the trace of the Hessian, $\mathrm{Hess}_{\phi}^g$, of $\phi$ and $\nabla^g$ is the gradient associated with $g$.
\item{(Bakry-\'Emery Ricci tensor)} For any $n \le N \le \infty$, the \textit{$(N, w)$-Bakry-\'Emery Ricci tensor}, or the \textit{$N$-Bakry-\'Emery Ricci tensor} for short, denoted by $\mathrm{Ric}_{N, w}^g$, is defined by
\begin{equation}\label{eq:bericci}
\mathrm{Ric}_{N, w}^g:=\mathrm{Ric}^g+\mathrm{Hess}_w^g-\frac{dw\otimes dw}{N-n},
\end{equation}
where $\mathrm{Ric}^g$ is the Ricci tensor of $(M^n, g)$. Note that $\mathrm{Ric}_{N, w}^g:=\mathrm{Ric}^g$ and $w$ is always assumed to be a (locally) constant whenever we discuss the case when $N=n$. 
\end{enumerate}
\end{definition}
Recall that in the same setting as above, we have;
\begin{itemize}
    \item for all $\phi, \psi \in C^{\infty}(M^n)$,
\begin{equation}
    \int_{M^n}g(\nabla^g \phi, \nabla^g \psi) \di \mathrm{vol}^g_w=-\int_{M^n}\Delta^g_w\phi \cdot \psi \di \mathrm{vol}^g_w
\end{equation}
if at least one of $\phi$ and $\psi$ has compact support;
\item if $(M^n ,g)$ is complete, then $C_c^{\infty}(M^n)$ is dense in $H^{1,p}(M^n)$ for any $p \in [1, \infty)$. In particular, if $\phi \in C^{\infty}(M^n)$ satisfies $|\phi|+|\nabla^g \phi| +|\Delta^g_w\phi| \in L^2(M^n)$, then $\phi \in D(\Delta)$ in the sense of Definition \ref{def:lap} with $\Delta \phi=\Delta^g_w\phi$;
\item if $\mathrm{Ric}_{N, w}^g \ge K$, then for any $\phi \in C^{\infty}(M^n)$ we have
\begin{equation}\label{prop:besmooth}
\frac{1}{2}\Delta_w^g|\nabla^g \phi|^2 \ge \max \left\{ |\mathrm{Hess}_{\phi}^g|^2, \frac{(\Delta_w^g\phi)^2}{N}\right\} +g(\nabla^g \Delta_w^g\phi, \nabla^g \phi) +K|\nabla^g \phi|^2,
\end{equation}
see also Theorem \ref{thm:hessboch}, \cite[Proposition 4.21]{ErbarKuwadaSturm} and the remark below.
\end{itemize}
\begin{remark}\label{rem:converse}
    The pointwise Bochner inequality (\ref{prop:besmooth}) is a direct consequence of the \textit{Bochner identity}:
    \begin{equation}\label{eq:identi}
        \frac{1}{2}\Delta_w^g|\nabla^g \phi|^2 =|\mathrm{Hess}_{\phi}^g|^2+g(\nabla^g \Delta_w^g\phi, \nabla^g \phi)+\frac{(\Delta^g_w\phi-\Delta^g\phi)^2}{N-n}+\mathrm{Ric}_{N, w}^g(\nabla^g\phi, \nabla^g\phi)
    \end{equation}
    with elemental inequalities $|\mathrm{Hess}_{\phi}^g|^2 \ge \frac{(\Delta^g\phi)^2}{n}$ and 
    \begin{equation}
        \frac{\alpha^2}{a}+\frac{(\beta-\alpha)^2}{b} \ge \frac{\beta^2}{a+b},\quad \text{for all $a>0, b>0, \alpha \in \mathbb{R}$ and $\beta \in \mathbb{R}$.}
    \end{equation}
    Conversely, (\ref{prop:besmooth}), more weakly, the inequality
    \begin{equation}\label{prop:besmooth5}
\frac{1}{2}\Delta_w^g|\nabla^g \phi|^2 \ge  \frac{(\Delta_w^g\phi)^2}{N} +g(\nabla^g \Delta_w^g\phi, \nabla^g \phi) +K|\nabla^g \phi|^2,
\end{equation}
    implies $\mathrm{Ric}_{N, w}^g \ge K$. Though this is an well-known fact, let us give a proof for reader's convenience as follows.

    Fix $x \in M^n$. Thanks to the smooth convergence of the rescaled weighted spaces to the tangent cone at $x$, $\mathbb{R}^n$, we can find a harmonic coordinate $\Phi=(\phi_1, \ldots, \phi_n)$ around $x$ (thus $\Delta^g_w\phi_i=0$) with the vanishing Hessian at $x$ (thus, in particular, $\Delta^g\phi_i(x)=0$). For all $a_i \in \mathbb{R}$, letting $\psi:=\sum_ia_i\phi_i$, since (\ref{eq:identi}) yields
    \begin{equation}
        \frac{1}{2}\Delta^g_w|\nabla^g \psi|^2(x)=\mathrm{Ric}_{N, w}^g(\nabla^g\psi, \nabla^g\psi)(x),
    \end{equation}
    this together with (\ref{prop:besmooth5}) shows
    \begin{equation}
        \mathrm{Ric}_{N, w}^g(\nabla^g\psi, \nabla^g\psi)(x) \ge K|\nabla^g \psi|(x).
    \end{equation}
    This implies $\mathrm{Ric}_{N, w}^g \ge K$ because $a_i$ are arbitrary.
\end{remark}

By this observation, from the next subsection, we will use the notations, $g=\langle \cdot, \cdot \rangle, \nabla^g=\nabla, \Delta^g_w=\Delta$ and so on, coming from the framework of metric measure spaces, explained in the previous section.
\begin{lemma}\label{lem:S}
    Let $(M^n, g)$ be a (possibly incomplete) Riemannian manifold and let $w, h \in C^{\infty}(M^n)$. Then for any Lipschitz function $\phi$ with compact support, we have
    \begin{equation}
        \frac{1}{2}\int_{M^n}\phi^2|\mathrm{Hess}_h^g|^2\di \mathrm{vol}^g_w\le \int_{M^n}\left(2|\nabla^g h|^2|\nabla^g \phi|^2-\phi^2g(\nabla^g \Delta_w^g h, \nabla^g h)-K\phi^2|\nabla^g h|^2\right) \di \mathrm{vol}^g_w.
    \end{equation}
\end{lemma}
\begin{proof}
This is done when $\phi$ is smooth, due to the discussions around \cite[(6)]{S} together with (\ref{prop:besmooth}). For reader's convenience, let us recall it as follows.

Firstly it follows from (\ref{prop:besmooth}) that
\begin{equation}\label{eq:bochner2}
    \frac{1}{2}\int_{M^n}|\nabla^g h|^2\Delta^g_w \phi^2 \di \mathrm{vol}_w^g \ge \int_{M^n} \phi^2\left(|\mathrm{Hess}_h^g|^2+g(\nabla^g \Delta^g_wh, \nabla^g h)+K|\nabla^g h|^2\right)\di \mathrm{vol}_w^g.
\end{equation}
On the other hand, 
\begin{align}
    \int_{M^n}|\nabla^g h|^2\Delta^g_w \phi^2 \di \mathrm{vol}_w^g&=-4\int_{M^n}|\nabla^g h|\phi g(\nabla^g |\nabla^g h|, \nabla^g \phi)\di \mathrm{vol}_w^g \nonumber \\
    &\le \int_{M^n}\left(\phi^2 |\mathrm{Hess}_h^g|^2+|\nabla^g h|^2|\nabla^g \phi|^2 \right)\di \mathrm{vol}_w^g.
\end{align}
Combining this with (\ref{eq:bochner2}) completes the proof in the case when $\phi$ is smooth.

For general $\phi$, since any Lipschitz function with compact support on $M^n$ can be approximated by smooth functions with compact supports in the sense of $H^{1,p}$ on any compact subset for any $p<\infty$, we conclude.
\end{proof}
Finally let us discuss a relationship between smooth spaces and $\RCD$ spaces.
It is known from 
\cite[Theorem 1.1]{Han} that a complete Riemannian manifold $(M^n, g)$ of dimension $n$ with the smooth boundary $\partial M^n$ and a weighted measure $\mathrm{vol}^g_w$ is an $\RCD(K, \infty)$ space if and only if
 \begin{equation}
     \mathrm{Ric}^g_{\infty, w}\ge K, \quad \Pi_{\partial M^n}^g \ge 0,
 \end{equation}
 where $\Pi_{\partial M^n}^g$ denotes the second fundamental form of the boundary.  
 In particular, combining this with the corresponding finite dimensional Bakry-\'Emery Ricci tensor observed in \cite{Han2}, we can obtain the following (see also \cite[Corollary 2.6]{Han} and \cite[Theorem 2.9]{Ketterer2}), where the later result is due to \cite[Proposition 4.21]{ErbarKuwadaSturm}.
 \begin{theorem}[Characterization of RCD space in smooth framework]\label{thm:han}
     Let $(M^n, g)$ be a complete Riemannian manifold of dimension $n$ possibly with the smooth boundary $\partial M^n$.
     For all $w \in C^{\infty}(M^n)$, $K \in \mathbb{R}$ and $N \in [n, \infty]$, we have the following.
     \begin{enumerate}
         \item Assume that the boundary $\partial M^n$ exists. Then $(M^n, \dist^g, \mathrm{vol}^g_w)$ is an $\RCD(K, N)$ space if and only if  
     $\mathrm{Ric}^g_{N, w} \ge K$ and $\Pi_{\partial M^n}^g \ge 0$.
     \item Assume that $M^n$ has no boundary. Then $(M^n, \dist^g, \mathrm{vol}^g_w)$ is an $\RCD(K, N)$ space if and only if $\mathrm{Ric}_{N, w}^g \ge K$.
     \end{enumerate}
 \end{theorem}
 \subsection{Capacity}
 In order to define almost smooth spaces in the next subsection, let us define \textit{zero capacity} as follows. Note that  the capacity itself is well-defined for general metric measure spaces, and that it is compatible with our terminology under PI, see Chapter 6 of \cite{BjornBjorn} for the details.
 
 Let us fix a metric measure space $X$.
 \begin{definition}[Zero $p$-Capacity]\label{def:capzero}
     Let $A$ be a subset of $X$ and let $p \in [1, \infty)$. We say that $A$ has \textit{zero $p$-capacity} if for any open ball $B$, there exists a sequence of locally Lipschitz functions $\phi_i$ on $B$ such that $\phi_i$ is equal to $0$ on a neighborhood of $A \cap B$ and that $\phi_i \to 1$ in $H^{1, p}(B)$.
 \end{definition}
 It is worth mentioning that in the definition above;
 \begin{itemize}
     \item by taking a truncation, with no loss of generality, we can assume that $\phi_i$ as above satisfies $0 \le \phi_i \le 1$ (thus, in the sequel, we always consider such functions);
     \item if $A$ has zero $p$-capacity, then $A$ has also zero $q$-capacity for any $q \le p$;
     \item if $\meas_X(A)=0$ and $A$ is closed, then the last condition, $\phi_i \to 1$ in $H^{1, p}(B)$, is satisfied when $\phi_i \to 1$ holds on each compact subset of $B$ with 
     \begin{equation}
         \sup_i\int_B|\nabla \phi_i|^p\di \meas_X<\infty
     \end{equation}
     because of applying Mazur's lemma (see the proof of Proposition \ref{prop:capzero}).
 \end{itemize} 
 \begin{definition}[Codimension]\label{def:codimension}
     Let $A$ be a 
     subset of $X$ and let $1\le p<\infty$. We say that the codimension of $A$ is \textit{at least $p$} (in the sense of Minkowski) if 
     \begin{equation}
         \limsup_{\epsilon \to 0}\epsilon^{-p}\meas_X(B_{\epsilon}(A \cap B))<\infty
     \end{equation}
     holds for any open ball $B$ (recall that $B_{\epsilon}(A)$ denotes the $\epsilon$-open neighborhood of $A$ in $X$). 
 \end{definition}
 \begin{proposition}\label{prop:capzero}
     Let $A$ be a 
     subset of $X$ whose codimension is at least $p$. Then $A$ has zero $p$-capacity.
 \end{proposition}
 \begin{proof}
     For any $\epsilon>0$, find a Lipschitz function $\phi_{\epsilon}$ on $[0, \infty)$ satisfying that $0\le \phi_{\epsilon} \le 1$, that $\phi_{\epsilon}=1$ on $[0, \epsilon]$, that $|\nabla \phi_{\epsilon}| \le \epsilon^{-1}$, and that $\supp \phi_{\epsilon} = [0, 2\epsilon]$. Then letting $f_{\epsilon}(x):=\phi_{\epsilon} (\dist_X(A \cap B, x))$, we have 
     \begin{equation}
         \int_B|f_{\epsilon}|^p\di \meas_X \to 0
     \end{equation}
     and 
     \begin{equation}
         \int_B|\nabla f_{\epsilon}|^p\di \meas_X \le C\meas_X\left( B_{2\epsilon}(A \cap B)\right) \cdot \epsilon^{-p} \le C.
     \end{equation}
     Thus applying Mazur's lemma, after taking some finite convex combinations $g_{\epsilon_i}$ of $\{f_{\epsilon}\}_{\epsilon>0}$, for some $\epsilon_i \to 0$, we see that $g_{\epsilon_i}=1$ on a neighborhood of $A \cap B$ and that $g_{\epsilon_i} \to 0$ in $H^{1,p}(B)$. Thus, considering $1-g_{\epsilon_i}$, we conclude. 
 \end{proof}
 \begin{proposition}\label{prop:codi4}
     Let $A$ be a closed subset of $X$ whose codimension is at least $4$. If $X$ is a PI geodesic space with IH and a QL, then $D_c(\Delta, X \setminus A)$ is dense in $D(\Delta)$. 
 \end{proposition}
 \begin{proof}
 Fix $\phi \in D(\Delta)$. Our goal is to find a sequence $\phi_i \in D_c(\Delta, X \setminus A)$ with $\phi_i \to \phi$ in $H^{1,2}(X)$ and $\Delta \phi_i \to \Delta \phi$ in $L^2(X)$. 
 
 \textbf{Step 1}: \textit{We can assume that $\phi, \Delta \phi \in L^{\infty}(X)$.}

 The proof is as follows. For any $L \ge 1$, let us consider a truncation;
 \begin{equation}\label{eq:trunc}
 \phi^L:=\max\left\{ \min \left\{ \phi, L\right\}, -L\right\}.
 \end{equation}

 Then Corollary \ref{cor:lipreg} tells us that $|\nabla h_t\phi^L| \in L^{\infty}(X)$. Moreover the same arguments as in (\ref{eq:laplacianbound}) (or just applying Theorem \ref{thm:gaus}) allow us to conclude $|\Delta h_t\phi^L| \in L^{\infty}(X)$. Note that $h_t\phi^L \to h_t \phi$ in $D(\Delta)$ as $L \to \infty$ because of $\phi^L \to \phi$ in $L^2(X)$ with (\ref{eq:heatflowbound}), and that $h_t\phi \to \phi$ in $D(\Delta)$ as $t \to 0$ because of (2) of Proposition \ref{prop:heatconv}. Thus we have \textbf{Step 1} by a diagonal argument.

 \textbf{Step 2}: \textit{Furthermore we can assume that $\phi$ has compact support.}
 
Thanks to Theorem \ref{thm:cutoff} (see also Proposition \ref{rem:lapcut}), we can find 
a sequence of good cut-off functions $\rho_i \in D_c(X)$ with $\sup_i(\|\nabla \rho_i\|_{L^{\infty}} + \|\Delta \rho_i\|_{L^{\infty}})<\infty$ and $\rho_i \to 1$ on each compact subset of $X$. Since $\rho_i\phi$ converge weakly to $\phi$ in $D(\Delta)$, applying Mazur's lemma completes the proof. 

\textbf{Step 3}: \textit{Conclusion.}

 Find an open ball $B$ with $\supp \phi \subset B$. For any sufficiently small $\epsilon>0$, taking a maximal $\epsilon$-separated set of $\supp \phi \setminus B_{2\epsilon}(A)$ with Theorem \ref{thm:cutoff} (after a rescaling $\epsilon^{-1}\dist_X$ if necessary), we see that there exists a good cut-off function $\psi_{\epsilon} \in D(\Delta)$ such that $\epsilon^2|\Delta \psi_{\epsilon}| +\epsilon|\nabla \psi_{\epsilon}| \le C$, $\supp \psi_{\epsilon} \subset B$, that $\psi_{\epsilon} |_{\supp \phi \cap B_{\epsilon}(A\cap B)}=0$, and that $\psi_{\epsilon}|_{\supp \phi\setminus B_{2\epsilon}(A \cap B)}=1$.

 Letting $\phi_{\epsilon}:=\psi_{\epsilon}\phi \in D_c(\Delta)$, we have
 \begin{align}
     \int_X|\Delta \phi_{\epsilon}|^2\di \meas_X&\le C\int_B\left( \psi_{\epsilon}^2(\Delta \phi)^2+|\nabla \psi_{\epsilon}|^2|\nabla \phi|^2 +\phi^2(\Delta \psi_{\epsilon})^2\right)\di \meas_X \nonumber \\
     &\le C+C\int_B|\nabla \psi_{\epsilon}|^2\di \meas_X+C\int_B(\Delta \psi_{\epsilon})^2\di \meas_X \nonumber \\
     &\le C + C \meas_X(B_{2\epsilon}(A \cap B)) \epsilon^{-2}+ C \meas_X(B_{2\epsilon}(A \cap B)) \epsilon^{-4}\le C.
 \end{align}
Similarly we have $\int_X|\nabla \phi_{\epsilon}|^2\di \meas_X\le C$. Thus, we have the weak convergence of $\phi_{\epsilon} \to \phi$ in $D(\Delta)$. Therefore, applying Mazur's lemma completes the proof.
 \end{proof}
 \begin{remark}
In the same setting as in Proposition \ref{prop:codi4}, we can also prove; if $X \setminus A$ is smooth in the sense of Definition \ref{def:smooth}, which will be explained below, then $C^{\infty}_c(X\setminus A)$ is dense in $D(\Delta)$.
 \end{remark}
 \begin{corollary}\label{cor:charc}
     Let $X$ be a PI geodesic space with IH, the SL and a QL, and let $K \in \mathbb{R}, N\in [1, \infty]$. If a locally $\BE(K, N)$ open subset $V$ of $X$ satisfies that its complement $X \setminus V$ has at least codimension $4$, then $X$ is a $\RCD(K, N)$ space.
 \end{corollary}
 \begin{proof}
 Applying Proposition \ref{prop:codi4} to given $f \in D(\Delta)$, we have (\ref{eq:weakboch}) as $U=X$, where we immediately used multiplying good cut-off functions to $\phi$, localities (\ref{eq:localitygrad}) and (\ref{eq:localitylap}). Thus the conclusion is a direct consequence of 
  Proposition \ref{prop:charbe} and Theorem \ref{thm:LtoG}.
 \end{proof}
\subsection{Almost smooth space}
Let us fix a proper metric measure space $X$.
\begin{definition}[Smoothness]\label{def:smooth}
    We say that an open subset $U$ of $X$ is $n$-dimensionally \textit{smooth} if for any $x \in U$ there exist a (possibly incomplete) Riemannian manifold $(M^n, g)$ of dimension $n$ and an open neighborhood $V$ of $x$ in $U$ such that
        \begin{enumerate}
            \item{($C^{\infty}$-metric structure)} $V$ is isometric to $(M^n, g)$ as metric spaces;
            \item{($C^{\infty}$-measure structure)} after identifying $V$ with $M^n$, the reference measure $\meas_X$ on $V$ can be written as 
        \begin{equation}\label{eq:smoothweight}
            e^{-w}\di \haus^n
        \end{equation}
        for some smooth function $w$.
        \end{enumerate}
\end{definition}
In the same setting as in Definition \ref{def:smooth}, let us denote by $\mathrm{Ric}_{N, X}$ the $N$-Bakry-\'Emery Ricci curvature on $U$ under identifying $V$ with $M^n$. 

The following is a direct consequence of Theorem \ref{thm:han} with Remark \ref{rem:converse}.
\begin{proposition}\label{prop:almostsmoothrcd}
    For an $n$-dimensionally smooth open subset $U$ of $X$ and all $K \in \mathbb{R}$ and $N \in [n, \infty]$, we see that $\mathrm{Ric}_{N, X} \ge K$ holds on $U$ if and only if $U$ is a locally $\RCD(K, N)$ space. 
    \end{proposition}
    We are now in a position to introduce \textit{almost smooth metric measure spaces}, as being smooth except for a closed subset of zero $2$-capacity.
\begin{definition}[Almost smooth metric measure space]\label{def:alsm}
    We say that $X$ is $n$-dimensionally \textit{almost smooth} if there exist a closed set $\mathcal{S}=\mathcal{S}^{C^{\infty}}$, called the \textit{singular set} of $X$ (from the point of view of smoothness), such that the following hold.
    \begin{enumerate}
        \item{(Smoothness)} Let $\mathcal{R}=\mathcal{R}^{C^{\infty}}:=X \setminus \mathcal{S}$ (and call it the \textit{smooth part}, or the \textit{regular set} of $X$). Then $\mathcal{R}$ is smooth in the sense of Definition \ref{def:smooth}.
        \item{(Zero capacity)} $\mathcal{S}$ has zero $2$-capacity. 
    \end{enumerate}
    Moreover $X$ is said to be \textit{non-collapsed} if $w$ as in (\ref{eq:smoothweight}) is equal to $0$. 
\end{definition}
In general, the Hausdorff dimension, denoted by $\mathrm{dim}_HX$, of $X$ is at least $n$ if $X$ is an $n$-dimensional almost smooth metric measure space. However there exists an example such that $\mathrm{dim}_HX>n$ even if it is additionally an $\RCD(K,N)$ space, see \cite[Theorem A]{PW} (see also \cite[Theorem A]{Pan} and \cite[Theorem 3.5]{DHPW}).
\begin{remark}[Example of PI, {\color{blue}SL}, flat, but not RCD]\label{rem:example}
Let us provide a globally PI almost smooth non-collapsed metric measure space $Z$ whose metric structure is flat on the smooth part, however $Z$ is not a locally $\BE(K, \infty)$ space for any $K$, thus, in particular, it is not an $\RCD(K, N)$ space for all $K, N$.

Firstly recall the definition of the metric cone $C(Y)$ over a metric space $Y$; 
\begin{equation}\label{eq:conenotation}
C(Y):=([0, \infty)\times Y)/(\{0\}\times Y),\quad p_{C(Y)}:=[(0, y)]
\end{equation}
where the distance $\dist_{C(Y)}$ is defined by
\begin{equation}
    \dist_{C(Y)}\left((t_1, y_1), (t_2, y_2)\right):=\sqrt{t_1^2+t_2^2-2t_1t_2\cos \left(\min \{\dist_Y(y_1, y_2), \pi\}\right)}.
\end{equation}
For any $r>0$, let $\mathbb{S}^1(r):=\{x \in \mathbb{R}^2| |x| =r\}$ equipped with the standard Riemannian metric $g_{\mathbb{S}^1(r)}$.
Then fixing  $\epsilon>0$, put $Z:=C(\mathbb{S}^1(1+\epsilon))$ equipped with the Hausdorff measure $\haus^2$ of dimension $2$. Let us prove that $Z$ satisfies the desired properties as follows, where note that $Z$ is the metric completion of $(0, \infty) \times \mathbb{S}^1(1+\epsilon)$ equipped with the cone metric $(dr)^2+r^2g_{\mathbb{S}^1(1+\epsilon)}$. See also \cite[Section 1]{CheegerColding}.

Notice that $Z$ is flat except for the pole $p=p_{C(\mathbb{S}^1(1+\epsilon))}$, and that a singleton $\{p\}$ has at least codimension $2$. In particular by Proposition \ref{prop:capzero} we know that $Z$ is $2$-dimensional almost smooth non-collapsed metric measure space. It is trivial by definition that $Z$ is a proper geodesic space. It is also worth mentioning that $Z \setminus \{p\}$ is weakly convex (see \cite[Definition 6.4]{Deng} for the definition).

Defining a map $f:\mathbb{S}^1(1+\epsilon) \to \mathbb{S}^1(1)$ by
$f(x):=\frac{1}{1+\epsilon}x$,
consider the induced bijective map $\phi_f:Z=C(\mathbb{S}^1(1+\epsilon)) \to C(\mathbb{S}^1(1)) \cong \mathbb{R}^2$ defined by
$
    \phi_f(t, y):=\left(t, f(y)\right)
$
under the notation (\ref{eq:conenotation}).
Then it is easy to see that $\phi_f$ is bi-Lipschitz because, for instance, since the differential $D\phi_f:T_xZ \to T_{\phi_f(x)}\mathbb{R}^2$ at any smooth point $x \in Z \setminus \{p\}$ satisfies $\frac{1}{1+\epsilon}|v| \le |D\phi_f(v)| \le |v|$ for any $v \in T_xZ$, comparing the lengths of curves of $\gamma:[0, 1] \to Z$ and of $\phi_f \circ \gamma:[0, 1] \to \mathbb{R}^2$, we conclude.

Recalling the stability of a globally PI condition under a metric measure bi-Lipschitz map (see Remark \ref{rem:bilipinv}), we know that $Z$ is globally PI. In particular, any $f \in H^{1,2}(Z)$ with $|\nabla f| \le 1$ for $\meas_Z$-a.e. has the Lipschitz representative (see Remark \ref{liprepresentative}). Thus recalling the weak convexity of $Z \setminus \{p\}$, we know that such $f$ has the $1$-Lipschitz representative, namely $Z$ verifies the SL.

It is easy to see by the splitting theorem \cite[Theorem 1.4]{Giglis} that $Z$ is not an $\RCD(0, N)$ space for all $N<\infty$.
Together with the scale invariance of $Z$, we see that $Z$ is not an $\RCD(K, N)$ space for all $K, N<\infty$. Next let us also check that $Z$ is not an $\RCD(K, \infty)$ space for any $K$ by contradiction. Assume that $Z$ is an $\RCD(K, \infty)$ space for some $K$. Recalling (\ref{eq:localitylap}) with the smoothness of $Z \setminus \{p\}$, it is easy to see that (\ref{eq:hesstr}) is valid. In particular, since $|\mathrm{Hess}_g|^2\ge \frac{(\Delta f)^2}{2}$, together with (\ref{eq:hessbochner}), we conclude that $Z$ is an $\RCD(K, 2)$ space. This is a contradiction.

Thus, thanks to Theorem \ref{thm:almosttoRCD}, which will be proved in the next section, we immediately know that any QL does not hold on $Z$. However we will be able to give a direct proof of it as follows.

Let $\psi \in C^{\infty}(\mathbb{S}^1(1+\epsilon))$ be an eigenfunction of the first positive eigenvalue $\lambda_1$ of the minus Laplacian, thus $\lambda_1=\frac{1}{(1+\epsilon)^2}$.
Consider a function $h$ on $Z$ defined by
\begin{equation}
    h(t, y)=t^{\lambda_1}\psi(y)=t^{\frac{1}{(1+\epsilon)^2}}\psi(y),
\end{equation}
under the same notation (\ref{eq:conenotation}).
Note that this must be harmonic on $Z$ except for the pole, thus on whole $Z$ by Proposition \ref{prop:capzero} (see also \cite[Theorem 1.67]{CM1} and \cite[the proof of Theorem D]{P}). However $h$ is not Lipschitz around the pole.

Finally let us check that $Z$ is not a locally $\BE(K, N)$ space by contradiction. Thus assume that $Z$ is a locally $\BE(K,N)$ space. Take $\phi \in D(\Delta) \cap L^{\infty}(Z)$ with $\phi \ge 0$ and $\Delta \phi \in L^{\infty}(Z)$, and fix $f \in D(\Delta)$ with $\Delta f \in H^{1,2}(Z)$. Find $\rho_R \in C^{\infty}_c(\mathbb{R})$ with $0 \le \rho_R \le 1$, $\rho_R=1$ on $(-R, R)$, and $\|\frac{\di \rho_R}{\di x}\|_{L^{\infty}}+\|\frac{\di^2\rho_R}{\di x^2}\|_{L^{\infty}} \to 0$ as $R \to \infty$. Letting
\begin{equation}
    \phi_R(z):=\rho_R (\dist_Z(p, z)),
\end{equation}
it is easy to see $\phi_R\phi \in D_c(\Delta) \cap L^{\infty}(Z)$ with $\limsup_{R\to \infty}\|\Delta(\phi_R\phi)\|_{L^{\infty}}<\infty$ because of the smoothness of $Z \setminus \{p\}$ with the elliptic regularity results on the smooth part. Therefore, in the local Bochner inequality;
\begin{equation}
    \frac{1}{2}\int_Z\Delta(\phi_R\phi)|\nabla f|^2\di \meas_Z \ge \int_Z\phi_R\phi \left( \frac{(\Delta f)^2}{N}+\langle \nabla \Delta f, \nabla f\rangle +K|\nabla f|^2\right)\di \meas_Z,
\end{equation}
letting $R \to \infty$ together with the dominated convergence theorem proves the global Bochner inequality;
\begin{equation}
    \frac{1}{2}\int_Z\Delta\phi|\nabla f|^2\di \meas_Z \ge \int_Z\phi \left( \frac{(\Delta f)^2}{N}+\langle \nabla \Delta f, \nabla f\rangle +K|\nabla f|^2\right)\di \meas_Z.
\end{equation}
Thus $Z$ must be an $\RCD(K, N)$ space because of (\ref{eq:expgrowth}) and Proposition \ref{prop:PIStoL}. This is a contradiction.

This observation should be compared with \cite[Question 4.13]{SZ}.
\end{remark}
\section{From almost smooth to RCD}\label{sec:5}
In this section we will discuss when an almost smooth metric measure space is an $\RCD$ space. 
\begin{theorem}\label{thm:be}
    Let $N \ge n$ and let $X$ be an $n$-dimensional almost smooth metric measure space with 
    \begin{equation}
        \mathrm{Ric}_{N, X}\ge K,\quad \text{on $\mathcal{R}$.}
    \end{equation}
    Assume that $X$ is proper and that the following regularity for the heat flow holds for a dense subset $D$ of $L^2(X)$: 
    \begin{itemize}
        \item for all $f \in D$ 
        and $0<t \le 1$, we have $|\nabla h_tf| \in L^{\infty}_{\mathrm{loc}}(X)$.
    \end{itemize} 
    Then $X$ is a locally $\BE(K, N)$ space. Furthermore if
    \begin{itemize}
        \item for all $f \in D$ 
        and $0<t \le 1$, we have $|\nabla h_tf| \in L^{\infty}(X)$,
    \end{itemize} 
    then $X$ is a $\BE(K,N)$ space.
\end{theorem}
\begin{proof}
Firstly let us check the first statement about the local BE result.
For any $f \in D(\Delta)$ with $f \in H^{1,2}(X)$, find a sequence of $f_j \in D$ converging to $f$ in $L^2(X)$, 
and then consider $h_tf_j$ for $t \le 1$. 
For any open ball $B$ and any Lipschitz cut-off function $\rho \in \mathrm{Lip}_c(2B)$ with $\rho=1$ on $B$, we can take a sequence of Lipschitz functions $\phi_{i} \in \mathrm{Lip}_c(\mathcal{R})$ as in Definition \ref{def:capzero} with $A=\mathcal{S}$ (for $2B$).

Then Lemma \ref{lem:S}  shows
\begin{align}
    &\frac{1}{2}\int_X(\rho\phi_{i})^2|\mathrm{Hess}_{h_tf_j}|^2\di \meas_X \nonumber \\
    &\le \int_X\left(|\nabla h_tf_j|^2|\nabla (\rho\phi_{i})|^2-(\rho\phi_{i})^2\langle \nabla \Delta h_tf_j, \nabla h_tf_j\rangle - K(\rho\phi_{i})^2|\nabla h_tf_j|^2\right)\di \meas_X.
\end{align}
Thus letting  $i \to \infty$ yields
\begin{equation}
    \int_B|\mathrm{Hess}_{h_tf_j}|^2\di \meas_X<\infty,
\end{equation}
therefore
\begin{equation}\label{eq:localh12}
|\nabla h_tf_j|^2 \in H^{1,2}_{\mathrm{loc}}(X)    
\end{equation}
because $B$ is arbitrary (see also Remark \ref{rem:sobolev}), where we immediately used our assumption; $|\nabla h_tf_j| \in L^{\infty}_{\mathrm{loc}}(X)$.

Fix $\phi \in D_c(\Delta) \cap L^{\infty}(X)$ with $\phi \ge 0$ and $\Delta \phi \in L^{\infty}(X)$, take an open ball $B$ with $\supp \phi \subset B$, and similarly, consider $\phi_i$ as above.
Since (\ref{prop:besmooth}) yields
\begin{align}
    &-\frac{1}{2}\int_X\langle \nabla (\phi \phi_{i}), \nabla |\nabla h_tf_j|^2\rangle \di \meas_X \nonumber \\
    &\ge \int_X\phi \phi_{i}\left(\frac{(\Delta h_tf_j)^2}{N}+\langle \nabla \Delta h_tf_j, \nabla h_tf_j\rangle +K|\nabla h_tf_j|^2\right)\di \meas_X, 
\end{align}
letting $i \to \infty$ with the dominated convergence theorem shows
\begin{align}\label{eq:bochnerheat1}
\frac{1}{2}\int_X\Delta \phi |\nabla h_tf_j|^2\di \meas_X 
    &=-\frac{1}{2}\int_X\langle \nabla \phi, \nabla |\nabla h_tf_j|^2\rangle \di \meas_X, \quad \text{(by (\ref{eq:localh12}))} \nonumber \\
    &=-\lim_{i \to \infty}\frac{1}{2}\int_X\langle \nabla (\phi \phi_{i}), \nabla |\nabla h_tf_j|^2\rangle \di \meas_X \nonumber \\
    &\ge \lim_{i \to \infty}\int_X\phi \phi_{i}\left(\frac{(\Delta h_tf_j)^2}{N}+\langle \nabla \Delta h_tf_j, \nabla h_tf_j\rangle +K|\nabla h_tf_j|^2\right)\di \meas_X \nonumber \\
    &= \int_X\phi \left(\frac{(\Delta h_tf_j)^2}{N}+\langle \nabla \Delta h_tf_j, \nabla h_tf_j\rangle +K|\nabla h_tf_j|^2\right)\di \meas_X.
\end{align}
Under taking the limit $j \to \infty$, let us recall that $h_tf_j \to h_tf$ in $H^{1,2}(X)$ and that $\Delta h_tf_j \to \Delta h_t f$ in $H^{1,2}(X)$ as observed in the proof of \textbf{Step 1} in the proof of Proposition \ref{prop:codi4}. Thus, taking $j \to \infty$ in (\ref{eq:bochnerheat1}) proves
\begin{align}\label{eq:bochnerheat}
\frac{1}{2}\int_X\Delta \phi |\nabla h_tf|^2\di \meas_X \ge \int_X\phi \left(\frac{(\Delta h_tf)^2}{N}+\langle \nabla \Delta h_tf, \nabla h_tf\rangle +K|\nabla h_tf|^2\right)\di \meas_X.
\end{align}
Thus letting $t \to 0$ completes the proof of the first statement.

In order to prove the second statement, under assuming the second heat flow regularity, we remark that the arguments above allow us to conclude that under the same notations,
\begin{equation}
    \frac{1}{2}\int_X\rho^2|\mathrm{Hess}_{h_tf_j}|^2\di \meas_X\le \int_X\left(|\nabla h_tf_j|^2|\nabla \rho|^2-\rho^2\langle \nabla \Delta h_tf_j, \nabla h_tf_j\rangle-K\rho^2|\nabla h_tf_j|^2\right)\di \meas_X
\end{equation}
holds for any $\rho \in \mathrm{Lip}_c(X)$. In particular we have 
\begin{equation}
    |\mathrm{Hess}_{h_tf_j}| \in L^2(X),
\end{equation}
thus, $|\nabla h_tf_j|^2 \in H^{1,2}(X)$. For any $\rho \in \mathrm{Lip}_c(X)$ with $\rho \ge 0$, finding an open ball $B$ with $\supp \rho \subset B$ and considering $\phi_i$ as in Definition \ref{def:capzero}, we have
\begin{equation}
    -\frac{1}{2}\int_X\langle \nabla (\rho \phi_i \phi), \nabla |\nabla h_tf_j|^2\rangle \di \meas_X \ge \int_X\rho \phi_i \phi\left(\frac{(\Delta h_tf_j)^2}{N}+\langle \nabla \Delta h_tf_j, \nabla h_tf_j\rangle +K|\nabla h_tf_j|^2\right)\di \meas_X.
\end{equation}
Letting $i \to \infty$, and taking $\rho$ as $\rho=1$ on a large ball with small gradient (as observed in the proof of Proposition \ref{rem:lapcut}) show
\begin{equation}
    -\frac{1}{2}\int_X\langle \nabla \phi, \nabla |\nabla h_tf_j|^2\rangle \di \meas_X \ge \int_X \phi\left(\frac{(\Delta h_tf_j)^2}{N}+\langle \nabla \Delta h_tf_j, \nabla h_tf_j\rangle +K|\nabla h_tf_j|^2\right)\di \meas_X.
\end{equation}
Since the left-hand-side is equal to (by a fact that $|\nabla h_tf_j|^2 \in H^{1,2}(X)$)
\begin{equation}
    \frac{1}{2}\int_X\Delta \phi |\nabla h_tf_j|^2\di \meas_X,
\end{equation}
letting $j \to \infty$ and then letting $t \to 0$ complete the proof of the second statement.
\end{proof}
\begin{remark}\label{rem:linfty}
    Of course, $D$ as in the theorem above can be taken as $L^{\infty} \cap L^2(X)$ because for any $f \in L^2(X)$, considering the truncation
\begin{equation}
    f^L:=\max\{\min\{f, L\},-L\} \in L^2 \cap L^{\infty}(X),
\end{equation}
then we know $f^L \to f$ in $L^2(X)$ as $L \to \infty$. This choice of $D$ plays a role later. 
\end{remark}
\begin{corollary}\label{cor:rcdc}
    Let $X$ be an $n$-dimensional almost smooth metric measure space with the Sobolev-to-Lipschitz property and the volume growth condition (\ref{eq:volumegrowth}). Then for any $K \in \mathbb{R}$ and any $N \in [n, \infty]$, the following two conditions are equivalent:
    \begin{enumerate}
        \item $X$ is an $\RCD(K, N)$ space;
        \item We have 
        \begin{equation}
            \mathrm{Ric}_{N, X}\ge K, \quad \text{on $\mathcal{R}$}
        \end{equation}
        with the following regularity property for the heat flow for a dense subset $D$ of $L^2(X)$:
        \begin{itemize}
            \item if $f \in D$, then $|\nabla h_tf| \in L^{\infty}(X)$ for any $0<t \le 1$.
        \end{itemize}
    \end{enumerate}
\end{corollary}
\begin{proof}
We already proved the implication from (2) to (1) by Theorem \ref{thm:be}. The other one is an easy consequence of the RCD theory, see, for instance, \cite[Theorem 3.1]{AmbrosioMondinoSavare2}, in fact the regularity of the heat flow is valid as $D=L^{\infty} \cap L^2(X)$ (recall Corollary \ref{cor:lipreg}). See also Remark \ref{rem:converse}.  
\end{proof}
The following gives a generalization of a main result \cite[Theorem 3.7]{Honda3} to the weighted framework including the non-compact case.
\begin{corollary}\label{thm:almostsmoothcompact}
    Let $X$ be an $n$-dimensional almost smooth metric measure space. Assume that 
    \begin{equation}
\mathrm{Ric}_{N, X}\ge K,\quad \text{on $\mathcal{R}$}
\end{equation}
holds for some $N \ge n$ and some $K \in \mathbb{R}$, that the canonical inclusion of $H^{1,2}(X)$ into $L^2(X)$ is a compact operator and that any eigenfunction $f$ of $-\Delta$ on $X$ satisfies $|\nabla f| \in L^{\infty}(X)$. Then $X$ is a $\BE(K, N)$ space.
\end{corollary}
\begin{proof}
Taking $D$ as the set of all linear combinations of eigenfunctions, apply Theorem \ref{thm:be} because $h_t\phi=e^{-\lambda t}\phi$ holds if $\phi$ is an eigenfunction of $-\Delta$ with the eigenvalue $\lambda$.
\end{proof}

Let us provide two different proofs of the following result because both of them have independent interests (though the later one is very short). The first one is an \textit{elliptic} way in the sense that we use eigenfunctions. The other one is a \textit{parabolic} way in the sense that we use the heat kernel. Note that Theorem \ref{thm:maincod2} is a direct consequence of the following together with Theorems \ref{thm:lpircd} and \ref{thm:qLCh}.
\begin{theorem}\label{thm:almosttoRCD}
Let $X$ be an $n$-dimensional almost smooth PI geodesic space with SL and a QL. If 
\begin{equation}
\mathrm{Ric}_{N, X}\ge K,\quad \text{on $\mathcal{R}$}
\end{equation}
for some $N \ge n$ and some $K \in \mathbb{R}$, then $X$ is an $\RCD(K, N)$ space.
\end{theorem}
\begin{proof}[The first (elliptic) proof]
Thanks to Theorem \ref{thm:LtoG}, it is enough to prove that $X$ is a locally $\BE(K, N)$ space. Moreover since Corollary \ref{thm:almostsmoothcompact} proves the desired result in the case when $X$ is compact, let us focus only on the case when $X$ is non-compact.
Then take:
\begin{enumerate}
    \item $\phi \in D_c(\Delta) \cap L^{\infty}(X)$ with $\phi \ge 0$ and $\Delta \phi \in L^{\infty}(X)$ (notice $\phi \in \mathrm{Lip}_c(X)$ because of Theorem \ref{thm:Poisson});
    \item $f \in D(\Delta)$ with $\Delta f \in H^{1,2}(X)$;
    \item an open ball $B$ with $\supp \phi \subset B$;
    \item a good cut-off $\rho \in D(\Delta)$ with $\rho=1$ on a neighborhood of $\supp \phi$, and $\supp \rho \subset B$ (note $\rho f \in D(\Delta)$);
    \item $\phi_i$ as in Definition \ref{def:capzero} with $A=\mathcal{S}$ for $B$.
\end{enumerate} 
Let us consider the spectral decomposition (recall (\ref{eq:sobolevpoincare}), where the non-compactness of $X$ plays a role here):
\begin{equation}\label{eq:spectral}
    \rho f=\sum_ia_i\psi_i, \quad \text{in $H^{1,2}_0(B)$}
\end{equation}
and 
\begin{equation}\label{eq:spectral2}
    \Delta (\rho f)=-\sum_ia_i\lambda_i\psi_i, \quad \text{in $L^{2}(B)$,}
\end{equation}
where 
\begin{equation}
    a_i:=\int_B\rho f \psi_i\di \meas_X
\end{equation}
and $\psi_i \in H^{1,2}_0(B) \cap D(\Delta, B)$ is the $i$-th eigenfunction of the $i$-th eigenvalue $\lambda_i$ of the (minus) Laplacian counted with multiplicities, associated with the Dirichlet boundary value condition on $B$, verifying $\int_B\psi_i\psi_j\di \meas_X=\delta_{ij}$ (see, for example, \cite[Appendix]{Honda3} for proofs of (\ref{eq:spectral}) and of (\ref{eq:spectral2}), though the appendix deals with the compact case, however the arguments work even in this setting). Let
\begin{equation}
    f_k:=\sum_i^ka_i\psi_i
\end{equation}
and note that $f_k$ is smooth on the regular part because of the elliptic regularity theorem.
Then Lemma \ref{lem:S} proves for any $\psi \in \mathrm{Lip}_c(B)$
\begin{align}
&\frac{1}{2}\int_B(\psi\phi_{i})^2|\mathrm{Hess}_{f_k}|^2\di \meas_X\nonumber \\
&\le \int_B\left(2|\nabla f_k|^2|\nabla (\psi\phi_{i})|^2-(\psi\phi_{i})^2\langle \nabla \Delta f_k, \nabla f_k\rangle -K(\psi\phi_{i})^2|\nabla f_k|^2\right)\di \meas_X.
\end{align}
Thus letting $i \to \infty$ shows by the arbitrariness of $\psi$
\begin{align}\label{eq:hess}
|\mathrm{Hess}_{f_k}|\in L^2_{\mathrm{loc}}(B),
\end{align}
in particular $|\nabla f_k|^2 \in H^{1,2}_{\mathrm{loc}}(B)$ because Theorem \ref{thm:Poisson} shows $|\nabla f_k| \in L^{\infty}_{\mathrm{loc}}(B)$. 

On the other hand, since  (\ref{prop:besmooth}) yields
\begin{align}\label{eq:localbochner}
    -\frac{1}{2}\int_B\langle \nabla (\phi \phi_i), \nabla |\nabla f_k|^2\rangle \di \meas_X \ge \int_B\phi \phi_i\left(\frac{(\Delta f_k)^2}{N} +\langle \nabla \Delta f_k, \nabla f_k\rangle+K|\nabla f_k|^2\right) \di \meas_X,
\end{align}
the left-hand-side of (\ref{eq:localbochner}) can be calculated as 
\begin{align}
    (\mathrm{LHS}) &\stackrel{i\to \infty}{\to}-\frac{1}{2}\int_B\langle \nabla \phi, \nabla |\nabla f_k|^2\rangle \di \meas_X \nonumber \\
    &=\frac{1}{2}\int_B\Delta \phi |\nabla f_k|^2\di \meas_X, \quad \text{(by $|\nabla f_k|^2 \in H^{1,2}_{\mathrm{loc}}(B)$ and $\supp \phi \subset  B$)} \nonumber \\
    &\stackrel{k\to \infty}{\to} \frac{1}{2}\int_B\Delta \phi |\nabla (\rho f)|^2\di \meas_X \nonumber \\
    &=\frac{1}{2}\int_B\Delta \phi |\nabla f|^2\di \meas_X,\quad \text{(by (\ref{eq:localitygrad})).}
\end{align}
Finally since the right-hand-side of (\ref{eq:localbochner}) can be easily calculated as
\begin{align}
    &(\mathrm{RHS})\nonumber \\
    & =\int_B\left(\phi \phi_i\frac{(\Delta f_k)^2}{N} - (\Delta f_k)^2\phi\phi_i - \Delta f_k\langle \nabla (\phi \phi_i), \nabla f_k\rangle +K\phi\phi_i|\nabla f_k|^2\right) \di \meas_X \nonumber \\
    &\stackrel{i \to \infty}{\to} \int_B\left(\phi \frac{(\Delta f_k)^2}{N} - (\Delta f_k)^2\phi - \Delta f_k\langle \nabla \phi, \nabla f_k\rangle +K\phi|\nabla f_k|^2\right) \di \meas_X \nonumber \\
    &\stackrel{k\to \infty}{\to}\int_B\left(\phi \frac{(\Delta (\rho f))^2}{N} - (\Delta (\rho f))^2\phi - \Delta (\rho f)\langle \nabla \phi, \nabla (\rho f)\rangle +K\phi|\nabla (\rho f)|^2\right) \di \meas_X \nonumber \\
    &=\int_B\left(\phi \frac{(\Delta f)^2}{N} - (\Delta f)^2\phi - \Delta f\langle \nabla \phi, \nabla f\rangle +K\phi|\nabla f|^2\right) \di \meas_X, \quad \text{(by (\ref{eq:localitygrad}) and (\ref{eq:localitylap})),} \nonumber \\
    &=\int_B\phi \left(\frac{(\Delta f)^2}{N} +\langle \nabla \Delta f, \nabla f\rangle +K|\nabla f|^2\right) \di \meas_X,
\end{align}
we conclude that $X$ is a locally $\BE(K, N)$ space. 
\end{proof}
\begin{proof}[The second parabolic proof]
This is a direct consequence of 
(\ref{eq:expgrowth}), Corollary \ref{cor:lipreg} and Theorem \ref{thm:be}.
\end{proof}
\begin{remark}
    In the above setting, the corresponding local result is also valid, namely; if an open subset $U$ of $X$ satisfies 
\begin{equation}
\mathrm{Ric}_{N, X}\ge K,\quad \text{on $U \cap \mathcal{R}$}
\end{equation}
for some $N \ge n$ and some $K \in \mathbb{R}$, then $U$ is a locally $\BE(K, N)$ space. It is worth pointing out that by the first proof above, this local result is also valid even if we replace PI with a QL by local ones which will be proposed later, Question \ref{ques:1}. Moreover if the singular set in $\bar U$ has at least codimension $4$, the local QL assumption can be dropped by the same arguments as in the theorem below, where the conclusion is still the same; $U$ is a locally $\BE(K, N)$ space. 
\end{remark}

In the rest of this section, we also provide two proofs of the following, as a special, but a typical version of Theorem \ref{thm:mainsmooth}. 
For the most general formulation, see Section \ref{sec:almostrcd}.

\begin{theorem}\label{thm:codimension4}
    Let $X$ be an $n$-dimensional almost smooth non-collapsed 
    metric measure space. Assume that the singular set $\mathcal{S}$ has at least codimension $4$ (recall Definition \ref{def:codimension})
     and that $X$ is a PI geodesic space with the SL and 
        \begin{equation}
        \mathrm{Ric}_{n, X} \ge K,\quad \text{on $\mathcal{R}$.}
    \end{equation}
    Then $X$ is a non-collapsed $\RCD(K, n)$ space.
\end{theorem}
In order to introduce the first proof, let us prepare the following (see also \cite[Theorem 2.6 and Lemma 2.7]{Dai}):
\begin{lemma}\label{lem:holder}
    Let $X$ be an $n$-dimensional almost smooth non-collapsed PI geodesic space with 
    \begin{equation}
        \mathrm{Ric}_{n, X} \ge K,\quad \text{on $\mathcal{R}$.}
    \end{equation}
    Then for any ball $B$ of radius $r \in (0, 1]$ with $2B \subset \mathcal{R}$ and any $f \in D(\Delta, 2B)$ with $|\Delta f| \le c$ on $2B$, we have
    \begin{equation}
        |\nabla f| \le C\left(\frac{1}{r}\intav_B|f|\di \meas_X +c\right),\quad \text{on $B$,}
    \end{equation}
    where $C>0$ depends only on $n, K$ and $c$.
\end{lemma}
\begin{proof} This is a direct consequence of Theorem \ref{thm:Poisson} (see \cite[Theorem 3.1 or Proposition 3.2]{JKY}) together with Cheng-Yau's gradient estimates on harmonic functions \cite[Theorem 6]{CY} because their proofs are completely local, after a rescaling $\frac{1}{r}\dist_X$ if necessary.
\end{proof}

\begin{corollary}\label{cor:holder}
    Under the same setting as in Lemma \ref{lem:holder}, we have for some $\gamma \in (0,1)$
    \begin{equation}\label{eq:desiredgrad}
        |\nabla f|(x) \le \frac{C}{\dist_X(x, \mathcal{S})^{1-\gamma}},\quad \text{for $\meas_X$-a.e. $x \in B\setminus \mathcal{S}$,}
    \end{equation}
    where $C>0$ depends only on $n, K$ and $c$.
\end{corollary}
\begin{proof}
    Firstly it follows from our PI assumption that $f$ is $\gamma$-H\"older continuous on $B$ for some $\gamma \in (0, 1)$ (see for instance \cite[Lemma 2.3]{JKY} for a quantitative estimate). Take $x \in B \setminus \mathcal{S}$ and put $s:=\frac{\dist_X(x, \mathcal{S})}{2}$. It is enough to show the gradient estimate (\ref{eq:desiredgrad}) for
    \begin{equation}
        f_s:=f-\intav_{B_s(x)}f\di \meas_X.
    \end{equation}
    Since $\int_{B_s(x)}f_s\di \meas_X=0$ we can find $y \in \bar B_s(x)$ with $f_s(y)=0$ because of the continuity of $f$. Thus applying Lemma \ref{lem:holder} we have
    \begin{align}
        |\nabla f|(x) &\le C\left(\frac{1}{s}\intav_{B_s(x)}|f_s|\di \meas_X+ c\right) \le C\left(\frac{1}{s}\dist(x,y)^{\gamma}+c\right)\le C\left( \frac{1}{s^{1-\gamma}}+c\right)
    \end{align}
    which completes the proof.
\end{proof}
\begin{remark}
    By the standard proof based on the Harnack inequality (see \cite[Theorem 5.13]{BMosco}), the H\"older exponent $\gamma$ above also can be taken as a quantitative one.
\end{remark}
We are now in a position to introduce the first proof of  Theorem \ref{thm:codimension4}.
\begin{proof}[The first proof of Theorem \ref{thm:codimension4}] 

Thanks to Theorem \ref{thm:almosttoRCD}, it is enough to prove a QL (see also Corollary \ref{cor:charc}).

Fix an open ball $B$ in $X$ of radius $r>0$ and a harmonic function $h:B \to \mathbb{R}$. Notice $|\nabla h| \in L^4(\frac{1}{2}B)$ because Corollary \ref{cor:holder} implies
    \begin{align}\label{al:l421}
    \int_{\frac{1}{2} B\cap B_1(\mathcal{S})}|\nabla h|^4\di \haus^n &=\sum_{j \ge 0}\int_{\frac{1}{2} B \cap (B_{2^{-j}}(\mathcal{S}) \setminus B_{2^{-j-1}}(\mathcal{S}))}|\nabla h|^4\di \haus^n \nonumber \\
    &\le C \sum_{j \ge 0}(2^{-j})^{4} \cdot (2^{j})^{4(1-\gamma)} \nonumber \\
    &=C\sum_{j \ge 0}(2^{-4\gamma})^j<\infty.
\end{align}
Lemma \ref{lem:S} yields
 \begin{equation}\label{eq:bbsh}
        \frac{1}{2}\int_{B}\phi^2|\mathrm{Hess}_h|^2\di \meas_X\le \int_{B}\left(2|\nabla h|^2|\nabla \phi|^2 -K\phi^2|\nabla h|^2\right) \di \meas_X
    \end{equation}
for any Lipschitz function $\phi$ on $B \cap \mathcal{R}$ with compact support. 
    Take $f_{\epsilon} \in \mathrm{Lip}(X)$ as in the proof of Proposition \ref{prop:capzero}, where we put $A=\mathcal{S}$, thus $|\nabla f_{\epsilon}| \in L^4(B)$.
Then considering any cut-off function $\rho \in \mathrm{Lip}_c(B)$ and applying (\ref{eq:bbsh}) to $\phi:=f_{\epsilon}\rho$ with letting $\epsilon \to 0$, we obtain 
    \begin{equation}
        |\mathrm{Hess}_h| \in L^2\left(\frac{1}{2}B\right).
    \end{equation}
    Thus combining this with (\ref{al:l421}), we have 
    \begin{equation}\label{eq:harmonich12}
        |\nabla h| \in H^{1,2}_{\mathrm{loc}}\left(\frac{1}{2}B\right).
    \end{equation}
On the other hand, the local Bochner inequality on $\mathcal{R}$ together with similar arguments as in the proof of \cite[Lemma 2.9]{Dai} allows us to conclude that 
\begin{equation}\label{eq:subhar}
    \Delta |\nabla h| -c|\nabla h| \ge 0
\end{equation}
holds for some $c>0$ in a weak sense on $\frac{1}{2}B\setminus \mathcal{S}$, namely 
\begin{equation}
    \int_{\frac{1}{2}B}\left(-\langle \nabla \phi, \nabla |\nabla h|\rangle -c\phi|\nabla h|\right) \di \meas_X \ge 0
\end{equation}
for any $\phi \in \mathrm{Lip}_c(\frac{1}{2}B\setminus \mathcal{S})$ with $\phi \ge 0$.
Recalling Proposition \ref{prop:capzero}, (\ref{eq:subhar}) is valid on $\frac{1}{2}B$ (across $\mathcal{S}$) in a weak sense. Thus 
the standard Moser iteration arguments (e.g. \cite{BMosco, BjornBjorn}) allow us to conclude 
\begin{equation}\label{eq:harm1}
    \||\nabla h|\|_{L^{\infty}(\frac{1}{4}B)}^2 \le \intav_{\frac{1}{2}B}|\nabla h|^2 \di \meas_X.
\end{equation}
Then the Caccioppoli inequality \cite[Lemma 3.3]{Jiang} with a quantitative $L^{\infty}$-estimate for harmonic functions \cite[Lemma 2.1]{JKY} for PI spaces shows
\begin{equation}\label{eq:harm2}
     \intav_{\frac{1}{2}B}|\nabla h|^2 \di \meas_X\le \frac{C}{r} \intav_{\frac{1}{2}B}|h|\di \meas_X.
\end{equation}
The inequalities (\ref{eq:harm1}) and (\ref{eq:harm2}) easily imply a QL. Thus we conclude.
\end{proof}

Let us provide the second proof of Theorem \ref{thm:codimension4}, establishing the  regularity on the heat flow requested in Corollary \ref{cor:rcdc}, directly, without using a QL.

\begin{proof}[The second proof of Theorem \ref{thm:codimension4}]
  Since $X$ is PI,   it has at most exponential volume growth by Proposition \ref{prop:poincaresobolev}. 
  To apply Theorem \ref{thm:be} it suffices to prove that for any $f\in L^2\cap L^\infty(X)$, we have $|\nabla h_tf|\in L^\infty(X)$ (see Remark \ref{rem:linfty}). 
  Using the heat kernel bounds \eqref{eq:heatflowbound} and \eqref{eqn:heat kernel bounds} we see that there exists a  $C>0$ such that for all $t>0$,
    $$\|h_tf\|_{L^\infty(X)}+\|h_t f\|_{L^2(X)}\leq C.$$
    Moreover, for any $j\geq 1$ there exists a $C_j>0$ such that  for all $t>0$
    \begin{equation}\label{eqn:heat kernel derivative bound}\|\partial_t^j h_tf\|_{L^\infty(X)}\leq C_{j}t^{-j/2}.
    \end{equation}
   Denote $u_t=|\nabla h_tf|$. Our goal is to show that $u_t\in L^\infty(X)$ for $t>0$.

 Let $\psi_\epsilon$ be a cut-off function which vanishes in the $\epsilon$-neighborhood of $\mathcal S$ and takes value 1 outside the $2\epsilon$-neighborhood of $\mathcal S$.   Given any ball $B\subset X$, let  $\rho$ be a Lipschitz cut-off function which is supported in $2B$ and takes value $1$ on $B$. Notice that $$\Delta(\partial_t (h_tf))=-\partial_t^2h_tf.$$ It follows from integration by parts that 
    $$\int_B (\rho\psi_\epsilon)^2|\nabla \partial_t h_tf |^2\di \meas_X\leq C \int_{2B} \left(|\nabla(\rho\psi_\epsilon)|^2(\partial_th_t f)^2+(\rho\psi_\epsilon)^2 (\partial_t^2h_tf)^2\right)\di \meas_X.$$
    Letting $\epsilon\rightarrow 0$, by \eqref{eq:difference} and \eqref{eqn:heat kernel derivative bound} we see that for $t>0$, $\partial_t h_tf\in H^{1, 2}(B)$. In particular,  $\partial_th_tf\in H^{1,2}_{\mathrm{loc}}(X)$.
    
   On $\mathcal R$ we have a refined Kato inequality (see the proof of \cite[Lemma A.1]{CW})
   \begin{equation}
       |\nabla^2h_tf|^2\geq \frac{n}{n-1}|\nabla u_t|^2-Ct^{-1/2}|\nabla u_t|.
\end{equation}
Fix $\delta>0$ small so that $q:=\frac{n-2}{n-1}+2\delta<1$. It follows  that 
$$|\nabla^2 h_tf|^2\geq \left(\frac{n}{n-1}-\delta\right)|\nabla u_t|^2-Ct^{-1}.$$
Denote $K_-=\min(K,0)$.
The Bochner formula implies the following two inequalities on $\mathcal R\times (0, \infty)$
\begin{equation} \label{eqn:elliptic inequality}
	\Delta u_t\geq u_t^{-1}|\nabla^2h_tf|^2-u_t^{-1}|\nabla u_t|^2+K_-u_t-|\nabla \partial_t h_tf|;
\end{equation}

\begin{equation}\label{eqn:parabolic inequality}
	(\partial_t+\Delta) u_t\geq u_t^{-1}|\nabla^2h_t f|^2-u_t^{-1}|\nabla u_t|^2+K_-u_t.
\end{equation}
A straightforward computation shows that
\begin{eqnarray*}
	\Delta u_t^{q}&=&qu_t^{q-1}\Delta u_t+q(q-1)u^{q-2}|\nabla u_t|^2\\
&\geq &q\delta u_t^{q-2}|\nabla u_t|^2-Ct^{-1} qu_t^{q-1}+qK_-u_t^q-qu_t^{q-1}|\nabla\partial_t h_tf|
\end{eqnarray*}
and
\begin{eqnarray*}
	(\partial_t+\Delta)u_t^{q}&=&qu_t^{q-1} (\partial_t+\Delta)u_t+q(q-1)u_t^{q-2}|\nabla u_t|^2\\
&\geq &q\delta u_t^{q-2}|\nabla u_t|^2-Ct^{-1} qu_t^{q-1}+qK_-u_t^q.
\end{eqnarray*}
Denote $v_t=u_t^q$ and $v(x, t)=v_t(x)$. 
  We claim $v\in H^{1, 2}_{\mathrm{loc}}(X\times (0, \infty))$. Notice that 
$$|\partial_t v|\leq q|\partial_t|\nabla h_tf||\leq q|\nabla\partial_th_tf|.$$
It follows that $\partial_tv\in L^2_{\mathrm{loc}}(X\times (0, \infty))$. 
One can see that on the regular set $\mathcal R\times (0,\infty)$ the following differential inequalities hold 
\begin{equation}\label{e:weak elliptic inequality}
	\Delta v_t\geq -Ct^{-1}v^{q-1}-Cv_t-qv_t^{\frac{q-1}{q}}|\nabla\partial_t h_tf|;
\end{equation}
\begin{equation}\label{e:weak parabolic inequality}
(\partial_t+\Delta)v\geq -Ct^{-1}v^{q-1}-Cv.
\end{equation}
As before we may multiple both sides of \eqref{e:weak elliptic inequality} by $(\rho\psi_\epsilon)^2v_t$ and integrate by parts to obtain
\begin{equation}\label{eqn:integration by parts}
	\int_{B} (\rho\psi_\epsilon)^2|\nabla v_t|^2\di \meas_X\leq C\int_{2B}(t^{-1}v_t^q+v_t^2+v_t^{2-\frac{1}{q}}|\nabla \partial_t h_tf|)\di \meas_X+\int_{2B} |\nabla (\rho\psi_\epsilon)|^2v_t^{2}\di \meas_X.
\end{equation}
 A standard application of the Li-Yau gradient estimate for heat flows implies that for all $t>0$,  $u_t=O(\dist_X(\mathcal S, \cdot)^{-1})$.  Since $q<1$ and $\|\partial_t u\|_{H^{1, 2}(X)}\leq C$, using our codimension $4$ assumption we see that the right hand side of \eqref{eqn:integration by parts} is uniformly bounded as $\epsilon\rightarrow 0$.  Therefore $v\in H^{1, 2}_{\mathrm{loc}}(X\times (0, \infty))$. This proves the claim.

It is now easy to check, again using the cut-off function $\rho\psi_\epsilon$ and our codimension $4$ assumption, that the parabolic differential inequality for $w$ holds weakly across the singular set $\mathcal S$. Since we have a uniform Sobolev inequality under our PI assumption, by parabolic Moser iteration we obtain that for any unit ball $B\subset X$ and $t>0$
$$\sup_{B}v_t\leq C_t\|v\|_{L^2(2B\times [t/2, 2t])}\leq C_t\|f\|_{L^2}^2.$$
In particular, $v_t\in L^\infty (X)$ for $t>0$. So $u_t=|\nabla h_tf|\in L^\infty(X)$ for $t>0$.
\end{proof}

\begin{remark}
The above results can be compared with the results of Chen-Wang \cite{CW}, where the authors studied the notion of a singular Calabi-Yau space and generalized the Cheeger-Colding theory to such spaces. In \cite{CW} most assumptions are local except the requirement on the weak convexity, which does not seem to be local and could be difficult to verify in practice. By contrast, our assumptions in Theorem \ref{thm:codimension4} are purely local and a version of weak convexity follows as a consequence from general RCD theory \cite[Theorem 6.5]{Deng}. One can show that a singular Calabi-Yau space in the sense of \cite{CW} is also an $\RCD(0, n)$ space because they actually checked a global PI condition with the SL and a QL, see \cite[Corollary 2.4, Propositions 2.7, 2.24 and Claim 3.16]{CW}.
\end{remark}
It is worth mentioning that the glued space of two $n$-dimensional spheres of radius $1$ at a common point is a $\BE(n-1, n)$ space (of course the regular part has constant sectional curvature $1$) with a globally volume doubling condition, but a local Poincar\'e inequality is not satisfied. See \cite[Example 3.8]{Honda3}. This observation allows us to conclude that the PI assumption in the next theorem is sharp.
Moreover recall that the corresponding $2$-dimensional statement does not hold, see Remark \ref{rem:example}.
\begin{theorem}[Characterization of Einstein $4$-orbifold]\label{thm:orbifold}
    Let $X$ be a $4$-dimensional 
    almost smooth non-collapsed metric measure space with
    \begin{equation}\label{eq:einstein}
        \mathrm{Ric}_{4, X}=K,\quad \text{on $\mathcal{R}$,}
    \end{equation}
    for some $K \in \mathbb{R}$. If the singular set is discrete,
then the following two conditions are equivalent. 
    \begin{enumerate}
        \item $X$ is PI with the SL and $Q=4$ as in (\ref{eq:ahlfors});
        \item $X$ is an Einstein orbifold.
    \end{enumerate}
\end{theorem}
\begin{proof}
    Since it is well-known that the implication from (2) to (1) holds (for instance techniques in \cite{CW} can be applied), let us prove the remaining one. Thus assume that (1) holds. 
    
    Then Theorem \ref{thm:codimension4} shows that $X$ is a non-collapsed $\RCD(K, 4)$ space.
    Take a point $x \in X$ and consider a tangent cone $Y$ at $x \in X$. Thanks to \cite[Proposition 2.8]{DG} (see also \cite{Ketterer}), we know $Y$ is the metric cone over some non-collapsed $\RCD(2,3)$ space $Z$. The smoothness of a neighborhood of $x$ except for $x$ implies the smoothness of $C(Z)$ except for the pole $p=p_{C(Z)}$, because of the Einstein equation of (\ref{eq:einstein}) (see also \cite[Theorem 7.3]{CheegerColding1} and \cite[Theorem 1.4]{CN}). Therefore we can conclude that $Z$ is also smooth  (where we immediately used the smoothness outside $\{p\}$ of the distance function $r_p$ from $p$, justified by the elliptic regularity theorem together with a fact that $\Delta r_p^2$ is a positive constant).  In particular $Z=\mathbb{S}^3/\Gamma$ with the sectional curvature $1$, where $\Gamma$ is a finite subgroup of $O(4)$. Since the set of all tangent cones at $x$ form a connected and compact set under the pointed Gromov-Hausdorff topology, it is easy to see that there is a unique tangent cone $Y$ at $x$. By the result of Bando-Kasue-Nakajima \cite{BKN} (see also \cite[Theorem 5.7]{DS} and \cite[Theorem 2.4]{Sesum}) one concludes that $X$ is an Einstein orbifold. 
\end{proof}

\section{Almost RCD spaces}\label{sec:almostrcd}
Notice that the smoothness assumption in the last section played  a role to get the $H^{1,2}$-estimate of $|\nabla f|^2$ with the $L^2$-bound of the Hessian of certain function $f$, coming from the Bochner inequality. Recalling the RCD theory, we can easily see that the smoothness can be replaced by RCD \textit{parts}, to get the same conclusion. Thus we reach the following definition.

\begin{definition}[Almost RCD space]\label{def:almostrcd}
    A metric measure space $X$ is said to be an \textit{almost} (\textit{non-collapsed}, respectively) $\RCD(K, N)$ \textit{space} for some $K \in \mathbb{R}$ and some $N \in [1, \infty]$ if there exist a closed set $\mathcal{S}=\mathcal{S}^{\RCD}$, called the RCD-\textit{singular set} of $X$ (from the point of view of being an RCD), such that the following hold.
    \begin{enumerate}
        \item{(RCD part)} Let $\mathcal{R}=\mathcal{R}^{\mathrm{RCD}}:=X \setminus \mathcal{S}$ and call it the RCD \textit{part}, or RCD-\textit{regular set} of $X$. Then $\mathcal{R}$ is a locally (non-collapsed, respectively) $\RCD(K, N)$ space (see Definition \ref{def:lrcd});
        \item{(Zero capacity)} $\mathcal{S}$ has zero $2$-capacity.
    \end{enumerate}
\end{definition}
Note that thanks to (\ref{eq:localitygrad}) with Remark \ref{rem:grad}, it is easy to see that $X$ is IH if it is an almost $\RCD$ space.

Let us introduce the following three results we can actually prove by the same ways as discussed in the last section.
\begin{theorem}\label{thm:fromalmostrcdtorcd}
    Let $X$ be an almost $\RCD(K, N)$ space whose metric structure is a length space for some $K \in \mathbb{R}$ and some $N \in [1, \infty)$. Then the following two conditions are equivalent.
    \begin{enumerate}
        \item $X$ is  PI with the SL and a QL.
        \item $X$ is an $\RCD(K, N)$ space.
    \end{enumerate}
\end{theorem}
\begin{theorem}
Let $X$ be an almost $\RCD(K, N)$ space for some $K \in \mathbb{R}$ and some $N \in [1, \infty]$.
Assume that $X$ is proper and that the following regularity for the heat flow holds for a dense subset $D$ of $L^2(X)$: 
    \begin{itemize}
        \item for all $f \in D$  
        and $0<t \le 1$, we have $|\nabla h_tf| \in L^{\infty}_{\mathrm{loc}}(X)$.
    \end{itemize} 
    Then $X$ is a locally $\BE(K, N)$ space. Furthermore if
    \begin{itemize}
        \item for all $f \in D$ 
        and $0<t \le 1$, we have $|\nabla h_tf| \in L^{\infty}(X)$,
    \end{itemize} 
    then $X$ is a $\BE(K,N)$ space.
\end{theorem}
\begin{corollary}\label{thm:almostcompact}
    Let $X$ be an almost $\RCD(K, N)$ space for some $K \in \mathbb{R}$ and some $N \in [1, \infty]$. Assume that the canonical inclusion of $H^{1,2}(X)$ into $L^2(X)$ is a compact operator and that any eigenfunction $f$ of $-\Delta$ on $X$ satisfies $|\nabla f| \in L^{\infty}(X)$. Then $X$ is a $\BE(K, N)$ space.
\end{corollary}
On the other hand, we need to take care of the proof of the following because a neighborhood of an RCD-regular point $x$, which can be identified by definition as an open subset of an $\RCD(K, N)$ space, may be far from the RCD-singular set $\mathcal{S}$. Recall that the corresponding almost smooth result, Theorem \ref{thm:codimension4}, was proved by using Lemma \ref{lem:holder} whose proof is based on a computation on a local coordinate established by Cheng-Yau, thus we could use the ball of radius $\frac{r}{4}$ to justify the arguments, where $r$ is the distance to $\mathcal{S}$.

Notice that a main result, Theorem \ref{thm:mainsmooth}, is a direct consequence of the following because of Proposition \ref{prop:almostsmoothrcd}.
\begin{theorem}\label{thm:cod4almost}
    Let $X$ be an almost $\RCD(K, N)$ PI geodesic space for some $K \in \mathbb{R}$ and some $N \in [1, \infty]$ with the SL. If $\mathcal{S}$ has at least codimension $4$,
     then $X$ is an $\RCD(K,N)$ space.

\end{theorem}
\begin{proof}
Let us follow the same line as in the proof of Theorem \ref{thm:codimension4} (see also Corollary \ref{cor:charc}). 
 Thus thanks to Theorem \ref{thm:fromalmostrcdtorcd}, it is enough to check that a QL holds. 

 \textbf{Step 1}: \textit{Consider an RCD-regular point $x \in \mathcal{R}$, put $r:=\dist_X(\mathcal{S}, x)>0$ and consider a harmonic function $h$ on $B=B_s(x)$ for some $0<s<r$. Then}
 \begin{equation}
     \|\nabla h\|_{L^{\infty}(\frac{1}{4}B)} \le \frac{C}{s}\intav_B|h|\di \meas_X.
 \end{equation}

 The proof is as follows. It follows from Definition \ref{def:almostrcd} and Theorem \ref{thm:qLCh} that 
 \begin{equation}\label{eq:lipco}
     |\nabla h| \in L^{\infty}_{\mathrm{loc}}(B).
 \end{equation}
On the other hand, similar arguments as in the proof of Lemma \ref{lem:S} allow us to show
 \begin{equation}
        \frac{1}{2}\int_{B}\phi^2|\mathrm{Hess}_h|^2\di \meas_X\le \int_{B}\left(2|\nabla h|^2|\nabla \phi|^2-K\phi^2|\nabla h|^2\right) \di \meas_X
    \end{equation}
for any Lipschitz function $\phi$ on $B (\subset \mathcal{R})$ with compact support. In particular we know $|\mathrm{Hess}_h| \in L^2_{\mathrm{loc}}(B)$, thus by (\ref{eq:lipco})
\begin{equation}
    |\nabla h|^2 \in H^{1,2}_{\mathrm{loc}}(B).
\end{equation}
Applying the local Bochner inequality, we have on $B$
\begin{equation}
    \Delta |\nabla h|^2 -2K|\nabla h|^2 \ge 0
\end{equation}
in a weak sense. 
Then the standard Moser iteration arguments show
\begin{equation}
    \|\nabla h\|_{L^{\infty}(\frac{1}{4}B)} \le C\intav_{\frac{1}{2}B}|\nabla h|^2 \di \meas_X,
\end{equation}
see for instance \cite[Lemma 5.13]{BPS} (though this is stated for $\RCD$ spaces, a locally PI condition is enough to get the same conclusion. See also \cite[Lemma 11.1]{Li}).
Then the desired conclusion comes from the above together with the Caccioppoli inequality (see for instance \cite[Lemma 3.3]{Jiang}) with a quantitative $L^{\infty}$-estimate for harmonic functions \cite[Lemma 2.1]{JKY}, as done in the first proof of Theorem \ref{thm:codimension4}. Thus we have \textbf{Step 1}.

In the sequel, let us fix an open ball $B$ in $X$ of radius $r>0$ and a harmonic function $h:B \to \mathbb{R}$. 

\textbf{Step 2}: \textit{We have for some $0<\alpha<1$} 
\begin{equation}\label{eq:holdercontiharmonic}
    |\nabla h|(x) \le \frac{C}{\dist_X(\mathcal{S}, x)^{1-\alpha}}. 
\end{equation}

    The proof is the same to that of Corollary \ref{cor:holder} combining with \textbf{Step 1}. 
    
\textbf{Step 3}: \textit{Conclusion.}

 The remaining arguments are completely the same to the corresponding parts of the proof of Theorem \ref{thm:codimension4}. Thus we conclude.
\end{proof}

Finally let us provide an observation on the \textit{ramified double cover} introduced in \cite{BBP} under assuming that the readers are familiar with this topic. 
\begin{remark}\label{rem:ramified}
Let $\pi :\hat X\to X$ be the ramified double cover of a non-collapsed $\RCD(K, N)$ space $X$ without boundary. It is conjectured by \cite[Question]{BBP} that $\hat X$ is also a non-collapsed $\RCD(K, N)$ space. 

Firstly note that a quantitative volume estimate for any ball $B$ of radius $1$ in $X$;
\begin{equation}
    \haus^N\left( B_r(\mathcal{S}^{N-2}_{\epsilon, r}) \cap B\right) \le Cr^2
\end{equation}
obtained in \cite[Theorems 1.7 and 1.9]{CJN} for Ricci limits should be generalized to $X$. We also refer the techniques in Sections 4 and 5 of \cite{BNS} along this direction. 

If so, then in particular we know that $X \setminus A$, following the terminologies in Section 3 of \cite{BBP}, should have zero $2$-capacity. Thus by the proof of \cite[Remark 3.2]{BBP} we will see that so is $\hat X \setminus \hat A$.

Then since it is easy to see the condition of RCD part for $\hat X \setminus \hat A$, we conclude that $\hat X$ is an almost $\RCD(K, N)$ space.

On the other hand, since we will be able to prove the local volume doubling condition on $\hat X$ by the same way as in the proof of (\ref{eq:expgrowth}) (see Remark \ref{rem:covering}), the remaining main issue to prove the conjecture above along this direction may be to show a local Poincar\'e inequality, the SL and a QL. 
This issue is, of course, related to the difficulty mentioned in \cite[Remark 1.24]{BBP}. 
\end{remark}

\section{Discussion and open problems}\label{sec:open}
Let us end the paper by providing a couple of questions related to this paper.

\begin{question}\label{ques:1}
Let $X$ be an IH metric measure space with the SL, whose metric structure is a length space, and let $U$ be a bounded open subset of $X$. Assume that $U$ is a locally $\BE(K, N)$ space for some $K \in \mathbb{R}$ and some $N \in [1, \infty)$, and that $U$ is \textit{locally PI} with a \textit{local QL} in the following sense;
\begin{enumerate}
    \item{(Local volume doubling on $U$)} There exists $c_{v, U}>1$ such that 
    \begin{equation}
        \meas_X(2B) \le c_{v, U}\meas_X(B) 
    \end{equation}
    holds for any ball $B$ with $2B \subset U$; 
    \item{(Local Poincar\'e inequality on $U$)} There exist $c_{p, U}>1$ and $\Lambda \ge 1$ such that for any Lipschitz function $f$ on $U$ and any ball $B$ of radius $r>0$ with $\Lambda  B \subset U$, we have
    \begin{equation}
        \intav_{B} \left| f-\intav_Bf\di \meas_X\right|\di \meas_X \le c_{p, U}r\left( \intav_{\Lambda  B}(\mathrm{Lip}f)^2\di \meas_X\right)^{\frac{1}{2}}.
    \end{equation}
    \item{(QL on $U$)} There exists $c_{h, U}>1$ such that for any harmonic function on a ball $B$ with $2B \subset U$, we have
    \begin{equation}
    \sup_{\frac{1}{2}B}|\nabla f| \le \frac{c_{h, U}}{r} \intav_{B}|f|\di \meas_X.
\end{equation}
\end{enumerate}
Then
  can we develop the structure theory of $U$ as in the Cheeger-Colding/RCD theories? For example, can we conclude that $U$ is a locally $\RCD(K, N)$ space? 
\end{question}
The Cheeger-Colding theory \cite{CheegerColding, CheegerColding1, CheegerColding2, CheegerColding3} (see also \cite{CJN,CN}) for Ricci limit spaces was developed by \textit{local} techniques, based on their \textit{almost splitting theorem} established in \cite[Theorem 6.62]{CheegerColding}. As counterparts, it is worth mentioning;
\begin{itemize}
    \item a result of Colding-Naber \cite[Theorem 1.1]{CN} for Ricci limit spaces was proved by \textit{global} techniques using the heat flow; 
    \item the structure theory of RCD spaces, for example \cite{BrueSemola, Deng, DG, MN}, are based on global techniques using, as a starting point, the convergence of the heat flows with respect to the pointed measured Gromov-Hausdorff convergence \cite[Theorem 5.7]{GigliMondinoSavare13} by Gigli-Mondino-Savar\'e.
\end{itemize}
Roughly speaking, the question above asks whether we can establish a similar structure theory for spaces with locally Ricci curvature bounded below. We refer \cite{Sturm95, Sturm96} as possible technical tools along this direction again. 

\begin{question}
    In Theorem \ref{thm:maincod2}, in order to get the same conclusion, is it possible to replace the QL-assumption, by the condition that any tangent cone is an $\RCD(0, N)$ space?
\end{question}
In connection with the question above, recall an example observed in Remark \ref{rem:example}.
\begin{question}\label{ques:convex}
    Let $X$ be an almost $\RCD(K, N)$ space. If the RCD-regular set is convex, then is $X$ an $\RCD(K, N)$ space?
\end{question}
Question \ref{ques:convex} should be fine if additionally, $X$ is assumed to be an almost smooth, and the RCD-singular set has at least codimension $4$, because of the same arguments as in \cite{CW}. 

The following is a Local-to-Global question. Recall Theorem \ref{thm:LtoG} and an example $Z$ observed in Remark \ref{rem:example} along this direction.
\begin{question}
    Let $X$ be a proper geodesic metric measure space with the Sobolev-to-Lipschitz property and the volume growth condition (\ref{eq:volumegrowth}). If $X$ is a locally $\BE(K, N)$ space, then is it an $\RCD(K, N)$ space?
\end{question}

\begin{question}\label{ques:convheat}
    Let $X_i$ be a pmGH convergent sequence of IH metric measure spaces to $X$. Assume that they are uniformly PI with a uniformly QL, in the sense; there exist $c_v, c_p, c_h$ such that all $X_i, X$ are PI for $c_v$ and $c_p$, with the QL for $c_h$. Then does the Mosco convergence of Cheeger energies hold as done in \cite{GigliMondinoSavare13} for RCD spaces? In particular, can we establish the convergence of the heat kernels?
\end{question}
Recalling techniques in the developments of the RCD theory (see for instance \cite{AmbrosioHonda, AmbrosioHonda2, BNS, BPS1, MN}), the \textit{local version} of Question \ref{ques:convheat} should be closely related to Question \ref{ques:1}.

Recall that it is proved in \cite[Theorem 1.1]{GKMS} that the quotient of an $\RCD$ space by an isometric action by a compact Lie group is also an $\RCD$ space.
The following asks for a kind of reverse direction of this. Note that \cite[Theorem 1.1]{MW} can be regarded as a positive result along this direction (see also Remark \ref{rem:covering}).
\begin{question}
Let $Y$ be a metric measure space with the SL and let $G$ be a compact Lie group isometrically (as metric measure spaces) acting on $Y$. If the quotient space $X=Y/G$ is an RCD space, then, when can we conclude that $Y$ is also an RCD space? For example, assume the following:
\begin{enumerate}
    \item 
    The quotient metric measure space $X=Y/G$ is an $\RCD$ space.
    \item 
    The projection $\pi:Y \to X$ is almost smooth in the sense; there exists a smooth part $\mathcal{R}$ of $X$ such that both $\pi^{-1}(\mathcal{R})$ and $\pi|_{\pi^{-1}(\mathcal{R})}$ are also smooth. 
    \item $Y\setminus \pi^{-1}(\mathcal{R})$ has  codimension  at least $4$.
    \item The $N$-Bakry-\'Emery Ricci tensor on the smooth part $\pi^{-1}(\mathcal{R})$ of $Y$ is bounded below.
\end{enumerate}
    Then can we conclude that $Y$ is an $\RCD$ space?
\end{question}
It is worth mentioning that the third assumption about the codimension of the singular set cannot be dropped because for the metric measure space $Z=C(\mathbb{S}^1(2))$ observed in Remark \ref{rem:example}, considering the canonical action of $\{\pm 1\}$ to $Z$ (thus the pole is the fixed point), though it is not a locally $\BE$ space, the quotient space $C(\mathbb{S}^1(1))\cong\mathbb{R}^2$ is an $\RCD(0, 2)$ space.
It is also interesting to ask whether the same conclusion holds after replacing (3) and (4) by being a locally $\BE$ space.

\end{document}